\documentclass{amsproc}

\usepackage[hidelinks,draft=false]{hyperref}
\hypersetup{
  colorlinks   = true, %Colours links instead of ugly boxes
  urlcolor     = blue, %Colour for external hyperlinks
  linkcolor    = blue, %Colour of internal links
  citecolor   = red %Colour of citations
}
\usepackage{amssymb}
\usepackage{mathrsfs}
\usepackage{dsfont}
\usepackage[nameinlink,noabbrev,capitalize]{cleveref}
\crefname{equation}{}{}
\usepackage{interval}
\intervalconfig{soft open fences}
\usepackage{enumitem}
\usepackage[nameinlink,noabbrev,capitalize]{cleveref}
\crefname{equation}{}{}
\newlist{theoenum}{enumerate}{1}
\setlist[theoenum]{label=\normalfont(\roman*), ref=\theproposition~\normalfont(\roman*), noitemsep, partopsep=0pt, topsep=0pt, parsep=0pt, itemsep=0pt}
\crefalias{theoenumi}{theorem}
\usepackage{url}

\theoremstyle{plain}
\newtheorem{theorem}{Theorem}
\newtheorem*{theorem*}{Theorem}
\newtheorem{proposition}[theorem]{Proposition}

\newtheorem*{lemma*}{Lemma}

\theoremstyle{definition}
\newtheorem{definition}[theorem]{Definition}

\theoremstyle{remark}

\newtheorem*{remark*}{Remark}

\theoremstyle{plain}
\newtheorem{openproblem}{Open Problem}

\newcommand*\diff{\mathop{}\!\mathrm{d}}
\newcommand{\R}{\mathds{R}}
\newcommand{\G}{\mathds{G}}
\newcommand{\C}{\mathds{C}}
\newcommand{\A}{\mathscr{A}}

\newcommand{\me}{\mathrm{e}}
\newcommand{\I}{\mathrm{i}}
\newcommand*{\tran}{^{\mkern-1.5mu\mathsf{T}}}

\renewcommand{\L}{\mathrm{L}}
\renewcommand{\exp}{\mathrm{exp}}
\renewcommand{\epsilon}{\varepsilon}
\DeclareMathOperator{\Lie}{\mathrm{Lie}}
\DeclareMathOperator{\T}{\mathrm{T}}
\DeclareMathOperator{\vspan}{\mathrm{span}}
\DeclareMathOperator{\Ima}{\mathrm{Im}}
\DeclareMathOperator{\Kern}{\mathrm{Ker}}
\DeclareMathAlphabet{\mymathbb}{U}{BOONDOX-ds}{m}{n}

\numberwithin{equation}{section}

\begin{document}

% \title[short text for running head]{full title}
\title[Normal forms for the exponential map of $\mathds{G}_\alpha$, $\mathrm{SU}(2)$, and $\mathrm{SL}(2)$]{Normal forms for the sub-Riemannian exponential map of $\mathds{G}_\alpha$, $\mathrm{SU}(2)$, and $\mathrm{SL}(2)$}

%    Only \author and \address are required; other information is
%    optional.  Remove any unused author tags.

%    author one information
% \author[short version for running head]{name for top of paper}
\author{Samuël Borza}
\address{Scuola Internazionale Superiore di Studi Avanzati (SISSA), via Bonomea 265, 34136 Trieste (TS), Italy}
\curraddr{Scuola Internazionale Superiore di Studi Avanzati (SISSA), via Bonomea 265, 34136 Trieste (TS), Italy}
\email{sborza@sissa.it}
\thanks{This project has received funding from the European Research Council (ERC) under the European Union’s Horizon 2020 research and innovation programme (grant agreement No. 945655).}

%    author two information
%\author{}
%\address{}
%\curraddr{}
%\email{}
%\thanks{}

%    The 2020 edition of the Mathematics Subject Classification is
%    the current definitive version.

\date{\today}

\begin{abstract}
The goal of this paper is to use singularity theory to find normal forms near the critical points of the sub-Riemannian exponential map.  Three cases are studied: the $\alpha$-Grushin plane with fold singularities, and the special unitary group $\mathrm{SU}(2)$ and special linear group $\mathrm{SL}(2)$ with fold and saddle-like singularities. They serve as examples of different sub-Riemannian structures and the techniques presented can be applied to other contexts. The paper also includes a discussion of the implications of this approach, as well as open problems. 
\end{abstract}

\maketitle

%    Text of article.

\section{Introduction}

The exponential map is a fundamental object used to study the geometry of Riemannian, Finsler, and sub-Riemannian manifolds. The exponential map plays a crucial role as it is related to geodesics, curvature, and other important geometric properties. Its behavior plays a crucial role in determining key analytic and topological characteristics of the underlying manifold.

The structure of conjugate locus, which is the set of singularities of the exponential map, is particularly important. In this paper, the exponential map is approached through singularity theory. The objective is to determine the normal forms of the exponential map in the neighbourhood of a critical point. We base our approach on the original work of \cite{warner1965}, as well as on \cite{borza2021}.

We deliberately choose to focus on three specific examples, the $\alpha$-Grushin plane, the three-dimensional special unitary group, and the three dimensional special linear group, in order to illustrate the relevant concepts and techniques. These examples are representative of different types of sub-Riemannian structure and the analysis presented can be adapted to other structures.

In \cref{sec:preliminaries}, the fundamentals of singularity theory and sub-Riemannian geometry will be presented, with a specific focus on Whitney folds, a significant type of singularity. Conjugate covectors, that is to say critical values of the exponential map,  will also be discussed in relation to Jacobi fields, and the relevant theory will be outlined.

The $\alpha$-Grushin plane is a specific type of almost Riemannian manifold that serves as a generalisation of the traditional Grushin plane. It is often considered a model example of a rank-varying sub-Riemannian manifold. We will show in \cref{sec:agrushin} that the exponential map of the $\alpha$-Grushin plane is equivalent to $f(x, y) = (x^2, y)$ in the neighbourhood of a singularity, see \cref{theo:normalformsGa}.

The three-dimensional special unitary group $\mathrm{SU}(2)$, on the other hand, is an example of 3D contact structure. It will be studied in \cref{sec:su2}. It is a Lie group equipped with a bi-invariant sub-Riemannian structure, just as the Heisenberg group is. The study will reveal that the conjugate locus of $\mathrm{SU}(2)$ is stratified by disconnected manifolds (see \cref{prop:conjSU2submanifold}). Additionally, it will be demonstrated in \cref{prop:normalformsSU2} that the exponential map of $\mathrm{SU}(2)$ can be represented as either $f(x, y, z) = (x z, y, z)$ or $f(x, y, z) = (x^2, y, z)$ depending on the nature of the intersection between the tangent space of the conjugate locus and the kernel of the exponential map. In \cref{sec:sl2}, the three dimensional special linear group $\mathrm{SL}(2)$ is handled in a similar manner.

The normal forms of the exponential map for these examples demonstrate that the exponential map does not exhibit the same behavior as $f(x) = x^3$ in the neighbourhood of a singularity, as shown in \cref{prop:localnoninjGa}, \cref{prop:localnoninjSU2} and \cref{prop:localnoninjSL2}. This property is well-known in Finsler geometry and is originally due to Morse and Littauer \cite{morselittauer1932} (see also \cite{warner1965} and the references therin for the historical development). It is not yet completely understood in sub-Riemannian geometry. \cref{sec:conclusion} delves into the relevant context, the challenges, and some open problems.

\section{Preliminaries}

\label{sec:preliminaries}

\subsection{Singularity theory and Whitney folds}

\label{subsec:whitneyfoldpleat}

Whitney's work \cite{whitney1955} provided the foundation for the study of singularities of smooth maps. Specifically, Whitney demonstrated that a generic smooth map from the plane into the plane can exhibit only two types of singularities, a fold or a cusp, up to a diffeomorphic change of coordinate. For readers interested in more advanced topics in singularity theory, we recommend the books \cite{guillemin}, and \cite{arnold1}.

We begin by defining what constitutes a ``good'' singularity according to Whitney.

\begin{definition}
\label{def:goodsingularity}

Let $f : M \to N$ be a smooth map between two manifolds $M$ and $N$ of equal dimensions. A point $p \in M$ is a singularity of $f$ if $\mathrm{Ker}(\diff_p f) \neq \{0\}$. The set of all singularities of $f$ is denoted by $\mathrm{Sing}(f)$. 

A singularity $p$ of $f$ is a \textit{good singularity}  if the following two additional
conditions are satisfied.
\begin{theoenum}
\item There exists a neighbourhood $\mathcal{U}$ of $p$ such that $0$ is a regular value of
\[
\mathcal{U} \to \R : x \mapsto \mathrm{det}(\diff_x f).
\]
\item For all $x \in \mathrm{Sing}(f) \cap \mathcal{U}$, we have that $\dim \mathrm{Ker}(\diff_x f) = 1$.
\end{theoenum}
\end{definition}

The above definition implies in particular that the set $\mathrm{Sing}(f) \cap \mathcal{U}$ of singularities of $f$ in the neighbourhood of $p$ is a submanifold of $M$ of codimension 1, by the implicit function theorem. 
\begin{definition}
\label{def:foldsingularity}
A good singularity of a smooth map $f : M \to N$ is a \textit{fold singularity} if
\[
\T_p(\mathrm{Sing}(f) \cap \mathcal{U}) \cap \mathrm{Ker}(\diff_x f) \neq \left\{0\right\} \text{ in } \T_p(M).
\]
\end{definition}

Whitney proved that, up to equivalence by smooth change of variables, the expression of a generic smooth map $f : M \to N$ in the neighbourhood of a good singularity $p \in M$ depends on whether $\mathrm{Ker}(\diff_p f)$ is tranversal to $\mathrm{Sing}(f) \cap \mathcal{U}$ or not. For the purposes of this work, we will only be addressing the transversal case here, which is the first type of singularity that emerges when classifying the singularities of a generic map (see \cite{arnold1}).

\begin{theorem}[{\cite[Section 15.]{whitney1955}}]
\label{theo:whitneyfoldpleat}
    If $p$ is a good and fold singularity of a smooth map $f : M \to N$ between manifolds $M$ and $N$ of same dimension $n$, then $f$ is equivalent to the map
    \begin{equation}
        \label{eq:whitneyfold}
        \R^n \to \R^n: (x, y_1, \dots, y_{n - 1}) \mapsto (x^2, y_1, \dots, y_{n - 1}).
    \end{equation}
    In other words, there exists a coordinate chart centered around $p$ in $M$ and a coordinate chart centered around $f(p)$ in $N$ such that the representation of $f$ in these coordinates is \cref{eq:whitneyfold}.
\end{theorem}

\subsection{Sub-Riemannian geometry and Jacobi equation}

\label{subsec:jacobifields}

A comprehensive reference for sub-Riemannian geometry can be found in \cite{comprehensive2020}. 

A sub-Riemannian structure on a smooth manifold $M$ is given by a set of $m$ global vector fields $X_1, \dots, X_m$. The distribution at $x \in M$ is the subspace of $\T_x(M)$ generated by the vector fields
\[
\mathcal{D}_x := \vspan \{X_1(x), \dots, X_m(x)\}.
\]
A horizontal (or admissible) curve $\gamma : \interval{a}{b} \to M$ is an absolutely continuous path such that there exists $u \in \mathrm{L}^2(\interval{a}{b}, \R^m)$ satisfying
\[
\dot{\gamma}(t) = \sum_{i = 1}^m u_i(t) X_i(\gamma(t)), \ \text{ for almost every } t \in \interval{a}{b}.
\]
We can define an inner product on $\mathcal{D}_x$  applying the polarisation formula to
\[
\langle v, v \rangle_x := \inf \left\{ \sum_{i = 1}^m u_i^2 \mid \sum_{i = 1}^m u_i X_i(x) = v \right\}.
\]
It can be proven that the map $t \mapsto \sqrt{\langle \dot{\gamma}(t), \dot{\gamma}(t) \rangle_{\gamma(t)}}$ is measurable if $\gamma$ is admissible. The sub-Riemannian length is then defined by
\[
\L(\gamma) = \int_{a}^{b} \lVert \dot{\gamma} (t) \rVert_{\gamma(t)} \diff t,
\]
while the sub-Riemannian distance of $M$ is
\[
	\diff(x, y) := \inf \{ \L(\gamma) \mid \gamma : \interval{a}{b} \to M \text{ is horizontal and } \gamma(a) = x \text{ and } \gamma(b) = y \}.
\]
It will always be assumed that a sub-Riemannian structures satisfies the \textit{Hörmander condition}, that is to say $\Lie_q(\mathcal{D}) = \T_q(M)$ for all $q \in M$. We also say in this case that $\mathcal{D}$ is \textit{bracket-generating}. This insures, by Chow–Rashevskii's theorem, that $\diff$ is a well-defined distance function and that the manifold and metric topology of $M$ coincide.

The Hamiltonian induced from the sub-Riemannian structure is the function
\[
H : \T^*(M) \to \mathds{R} : (q, \lambda_0) \mapsto \frac{1}{2}\sum_{k = 1}^m \langle \lambda_0, X_k(q) \rangle^2.
\]
We recall that the Hamiltonian vector field of a map $a \in C^{\infty}(\T^*(M))$ is the unique vector field $\overrightarrow{a}$ on $\T^*(M)$ that satisfies
\[
\sigma(\cdot, \overrightarrow{a}(\lambda)) = \diff_\lambda a, \qquad \forall \lambda \in \T^*(M),
\]
where $\sigma$ is the canonical symplectic structure on $\T^*(M)$. We also denote by $\pi : \T^*(M) \to M$ the bundle projection of $\T^*(M)$ into $M$.

In the examination of sub-Riemannian geodesics of $M$, which are locally minimising constant speed curves with respect to the sub-Riemannian distance, two types of geodesics are identified: abnormal and normal. A geodesic $\gamma(t)$ is normal if there exists a lift $\lambda(t) \in \T^*(M)$ such that $\pi(\lambda(t)) = \gamma(t)$ and satisfying
\begin{equation}
    \label{eq:hamiltoneqs}
    \dot{\lambda}(t) = \overrightarrow{H}(\lambda(t)).
\end{equation}
Conversely, the projection onto $M$ of any solution to \cref{eq:hamiltoneqs} in $\T^*(M)$ is a locally minimising path parametrised with constant speed $2 H$.

While the study of abnormal extremals is an interesting topic related to hard problems in sub-Riemannian geometry, this work will only focus on normal extremals, that is solution to Hamilton's equation \cref{eq:hamiltoneqs}. The flow of Hamilton's equation is denoted by $\me^{t \overrightarrow{H}}$.

Consider a normal extremal $\lambda(t)$ of a sub-Riemannian manifold $M$, that is to say
\[
\lambda(t) := \me^{t \overrightarrow{H}}(q, \lambda_0),
\]
for some initial condition $(q, \lambda_0) \in \T^*(M)$.
A Jacobi field $\mathcal{J}$ along $\lambda(t)$ is the variation field of a variation of $\lambda(t)$ through normal extremals. In other words, consider a variation of $\lambda(t)$ through normal extremals
\[
\Lambda(t, s) = \me^{t \overrightarrow{H}}(\Gamma(s)),
\]
where $\Gamma : \ointerval{-\epsilon}{\epsilon} \to \T^*(M)$ is a curve such that $\Gamma(0) = (q, \lambda_0)$, then the corresponding Jacobi field is
\[
\mathcal{J}(t) := \dfrac{\partial \Lambda}{\partial s}(t, 0) = \diff_{\lambda_0}\me^{t \overrightarrow{H}} [\dot{\Gamma}(0)].
\]
A Jacobi field $\mathcal{J}$ along $\lambda$ is therefore uniquely determined by its initial value $\dot{\Gamma}(0) \in \mathrm{T}_{(q, \lambda_0)}(\mathrm{T}^*(M))$. The dimension of the space of Jacobi fields along the normal geodesic $\lambda$ is 2$n$, where $n$ is the dimension of the manifold $M$.

Equivalently, it can be seen that a vector field $\mathcal{J}$ is a Jacobi field along the extremal $\lambda$ if and only if it satisfies
\begin{equation}
	\label{lieeq}
	\dot{\mathcal{J}} := \mathcal{L}_{\overrightarrow{H}} \mathcal{J} = 0,
\end{equation}
where $\mathcal{L}_{\overrightarrow{H}} \mathcal{J}$ is the Lie derivative of a vector field along $\lambda$ in the direction of $\overrightarrow{H}$:
\begin{equation}
    \label{eq:liederivatire}
    \mathcal{L}_{\overrightarrow{H}} \mathcal{J} (t) = \lim_{\epsilon \to 0} \dfrac{(\diff_{\lambda(t + \epsilon)} \me^{-\epsilon \overrightarrow{H}})[\mathcal{J}(t + \epsilon)] - \mathcal{J}(t)}{\epsilon} = \frac{\diff}{\diff\epsilon}\Big|_{\epsilon=0} (\diff_{\lambda(t + \epsilon)} \me^{-\epsilon \overrightarrow{H}})[\mathcal{J} (t + \epsilon)].
\end{equation}

We can re-write the differential equation \cref{lieeq} by using a symplectic frame along the curve $\lambda(t)$. A symplectic moving frame $(E, F) = (E_1, \dots, E_n, F_1, \dots, F_n)$ is a Darboux basis of $\T_{\lambda(t)}(\T^*(M))$ such that $\diff_{\lambda(t)} \pi[E_i] = 0$ for all $i = 1, \dots, n$.

The elements of a symplectic moving frame $(E, F)$ along $\lambda(t)$ always satisfy
\begin{equation}
    \label{eq:noncanonicalmovingframe}
    \begin{pmatrix}
			\dot{E} \\
			\dot{F}
		\end{pmatrix}
		=
		\begin{pmatrix}
			C_1(t)\tran & -C_2(t) \\
			R(t) & -C_1(t)
		\end{pmatrix}
		\begin{pmatrix}
			E \\
			F
    \end{pmatrix},
\end{equation}
for smooth curves of $n \times n$ matrices $C_1(t), C_2(t)$ and $R(t)$ such that $C_2(t)$ and $R(t)$ are symmetric, $C_2(t) \geq 0$. A vector field $\mathcal{J}(t) = \sum_{i = 1}^n p_i(t) E_i(t) + x_i(t) F_i(t)$ along $\lambda(t)$ is then a Jacobi field, i.e. it satisfies the differential equation \cref{lieeq}, if and only if
\begin{equation}
    \label{eq:jacobifields}
    \begin{pmatrix}
			\dot{p} \\
			\dot{x}
		\end{pmatrix}
		=
		\begin{pmatrix}
			- C_1(t) & -R(t) \\
			C_2(t) & C_1(t)\tran
		\end{pmatrix}
		\begin{pmatrix}
			p \\
			x
    \end{pmatrix}.
\end{equation}

The structural equation \cref{eq:noncanonicalmovingframe} can be made simpler by utilising a canonical symplectic frame along $\lambda(t)$, given that the normal extremal $\lambda(t)$ is \textit{ample} and \textit{equiregular}. We will not go into all the technical definitions here, but they can be found in references such as \cite{curvature} and \cite{zelenkoli2009}. It is sufficient for the present work to note that an ample and equiregular geodesic is associated with a Young diagram $D$. An ample and equiregular geodesic has a finite rank $k$, and row integers $n_1 \leq n_2 \leq \dots \leq n_k$ satisfying $\sum_{i = 1}^k n_i = \dim(M)$. The corresponding Young diagram is then
\[
D := \left\{ ai := (a, i) \mid a = 1, \dots, k \text{ and } i = 1, \dots, n_a\right\}.
\]

The following statement provides a summary of the structure of this canonical frame.

\begin{theorem}[\cite{zelenkoli2009}]
    \label{zelenkoli}
    If the normal extremal $\lambda(t)$ is ample and equiregular with Young diagram $D$, there exists a symplectic moving frame $(E_1, \dots, E_n, F_1, \dots, F_n)$ along $\lambda(t)$ such that the matrices $C_1(t)$, $C_2(t)$ in the structural equation \cref{eq:noncanonicalmovingframe} take the form
    \[
    C_1(t) =
    \begin{pmatrix}
    C_1(D_1) & & \\
   & \ddots & \\
   & & C_1(D_k)
    \end{pmatrix}, \ 
    C_2(t) =
    \begin{pmatrix}
    C_2(D_1) & & \\
   & \ddots & \\
   & & C_2(D_k)
    \end{pmatrix},
    \]
    where
    \[
    C_1(D_a) = 
    \begin{pmatrix}
    0 & \mymathbb{1}_{n_a -1}\\
    0 & 0
    \end{pmatrix}, \
    C_2(D_a) = 
    \begin{pmatrix}
    0 & 0\\
    0 & \mymathbb{0}_{n_a -1}
    \end{pmatrix}
    \]
    for each row $a = 1, \dots, k$ of the Young diagram $D$, and such that the matrix $R(t)$ is normal in the sense of \cite[Definition 1 and 2]{zelenkoli2009}. Furthermore, this symplectic moving frame is unique up to an change of basis that orthogonally preserves the Young diagram $D$, as stated in \cite[Theorem 1]{zelenkoli2009}.
\end{theorem}

The study of Jacobi fields will be crucial for understanding the conjugate locus.

\subsection{Conjugate locus in sub-Riemannian geometry}

The objective of this paper is to investigate the singularities of the exponential map for specific sub-Riemannian structures.

\begin{definition}
	The sub-Riemannian exponential map at $q \in M$ is the map
	\[
	\exp_q : \A_q \to M : \lambda \mapsto \pi(\me^{\overrightarrow{H}}(\lambda)),
	\]
	where $\A_q \subseteq \T^*_q(M)$ is the open set of covectors such that the corresponding solution of \cref{eq:hamiltoneqs} is defined on the whole interval $\interval{0}{1}$.
\end{definition}

The sub-Riemannian exponential map $\exp_q$ is smooth. If $\lambda : \interval{0}{T} \to \T^*(M)$ is the normal extremal that satisfies the initial condition $(q, \lambda_0) \in \T^*(M)$, then the corresponding normal extremal path $\gamma(t) = \pi(\lambda(t))$ by definition satisfies $\gamma(t) = \exp_q(t \lambda_0)$ for all $t \in \interval{0}{T}$. If $M$ is complete for the sub-Riemannian distance, then $\A_q = \T^*_q(M)$ and if in addition there are no stricly abnormal length minimisers, the exponential map $\exp_q$ is surjective. Contrary to the Riemannian case, the sub-Riemannian exponential map is not necessarily a diffeomorphism of a small ball in $\T^*_q(M)$ onto a small geodesic ball in $M$. In fact, $\Ima (\diff_0 \exp_q) = \mathcal{D}_q$ and $\exp_q$ is a local diffeomorphism at $0 \in \T^*_q(M)$ if and only if $\mathcal{D}_q = \T^*_q(M)$.

The set of singularities of $\exp_q$, as outlined in \cref{def:goodsingularity}, is known as the conjugate locus of a sub-Riemannian structure.

\begin{definition}
\label{subsec:conjugatelocus}
The conjugate locus at $q \in M$ is the set
\[
\mathrm{Conj}_q(M) := \left\{\lambda_0 \in \mathscr{A}_q \mid \mathrm{Ker}(\diff_{\lambda_0} \exp_q) \neq \{0\} \right\} \subseteq \T^*_q(M).
\]
The elements of $\mathrm{Conj}_q(M)$ are called conjugate covectors. We say that they are conjugate to the point $q$. The order of $\lambda_0 \in \mathrm{Conj}_q(M)$ is the dimension of $\mathrm{Ker}(\diff_{\lambda_0} \exp_q)$.
\end{definition}

It is important to note that in contrast to the usual approach, the conjugate locus to a point $q \in M$ is considered here as a subset of $\T^*_q(M)$ and not as the image in $M$ of the singularities of the exponential map.

The following statement clarifies the connection between conjugate covectors and Jacobi fields.

\begin{proposition}
    \label{prop:jacobivsconjpt}
	Let $\gamma : \interval{0}{1} \to M$ be a normal geodesic with initial covector $\lambda_0 \in \A_q$ and such that $\gamma(0) = q \in M$. The covector $\lambda_0$ is a critical point for $\exp_q$ if and only if there exists a non trivial Jacobi field $\mathcal{J}$ along $\gamma$ such that $\diff_{\lambda(0)}\pi[\mathcal{J}(0)] = 0$ and $\diff_{\lambda(1)}\pi[\mathcal{J}(1)] = 0$.
\end{proposition}

\cref{prop:jacobivsconjpt}, when expressed in terms of a symplectic frame, implies that a covector $\lambda_0 \in \T^*_q(M)$ is conjugate to $q \in M$ if there exists a non-trivial vector field $\mathcal{J}(t) = \sum_{i = 1}^n p_i(t) E_i(t) + x_i(t) F_i(t)$ along $\lambda(t)$ such that $(p(t), x(t))$ satisfies \cref{eq:jacobifields} and $x(0) = x(1) = 0$. 

When a symplectic moving frame is fixed along the extremal, we can make the following identifications. On one hand, the elements $F_1(t), \dots, F_n(t)$ of the symplectic frame can be interpreted as a basis for $\T_{\gamma(t)}(M)$ through the map $\diff_{\lambda(t)}\pi : \T_{\lambda(t)}(\T^*(M)) \to \T_{\gamma(t)}(M)$. On the other hand, the elements $E_1(t), \dots, E_n(t)$ can be interpreted as a basis for $\T_{\lambda(t)}(\T^*_{\gamma(t)}(M))$ through the differential of the inclusion $\T_{\gamma(t)}^*(M) \to \T^*(M) : \lambda \mapsto (\gamma(t), \lambda)$.

The kernel and the image of the differential of the exponential map can then also be represented using coordinates (see also \cite[Remark 22]{borza2021}):
\begin{equation}
    \label{eq:kernelconjugatelocus}
    \mathrm{Ker}(\diff_{\lambda_0} \exp_q) = \left\{ \sum_{i = 1}^n p_i(0) E_i(0) \mid (p(t), x(t)) \text{ satisfies } \cref{eq:jacobifields} \text{ and } x(0) = x(1) = 0 \right\},
\end{equation}
and
\begin{equation}
    \label{prop:imageexp}
    \mathrm{Im}(\diff_{\lambda_0} \exp_q) = \left\{ \sum_{i = 1}^n x_i(1) F_i(1) \mid (p(t), x(t)) \text{ satisfies } \cref{eq:jacobifields} \text{ with } x(0) = 0 \right\}.
\end{equation}

To conclude this section, we present an important result, which is also given in terms of the symplectic moving frame, that generalizes the regularity property of the Riemannian exponential map as investigated by Warner in \cite[Section 2. (R2)]{warner1965}.

\begin{proposition}[{Regularity property of the sub-Riemannian exponential map \cite[Proposition 21]{borza2021}}]
\label{prop:kernelexp}
The map
\[
\Kern(\diff_{\lambda_{0}} \exp_q) \to \T_{\exp_q(\lambda_{0})}(M)/\diff_{\lambda_{0}} \exp_q(\T_{\lambda_{0}}(\T^*_q(M)))
\]
sending $\sum_{i = 1}^n p_i(0) E_i(0)$ to $\sum_{i = 1}^n p_i(1) F_i(1) + \diff_{\lambda_{0}} \exp_q(\T_{\lambda_{0}}(\T^*_q(M)))$ is a linear isomorphism.
\end{proposition}

We are prepared to now analyse specific sub-Riemannian structures: the $\alpha$-Grushin plane, the three-dimensional special unitary group, and the three dimension special linear group.

\section{The \texorpdfstring{$\alpha$}{α}-Grushin plane}

\label{sec:agrushin}

For $\alpha \geq 1$, the $\alpha$-Grushin plane $\mathds{G}_\alpha$ is the sub-Riemannian structure on $\mathds{R}^2$ generated by the global vector fields $X = \partial_x$ and $Y_\alpha = |x|^\alpha \partial_y$. The geometry of $\mathds{G}_\alpha$ was studied for example in \cite{ligrushin} and in \cite{borza2022}. The $\alpha$-Grushin plane is ideal, i.e., it does not contain any non-trivial abnormal geodesics (see \cite[Lemma 4.2.1.]{borzathesis}). It is worth noting that this sub-Riemmanian structure is rank-varying: the dimension of $\mathcal{D}_{(x, y)} := \mathrm{span}\left\{ X(x, y), Y_\alpha(x, y) \right\}$ is 2 when $x \neq 0$ and 1 otherwise. The sub-Riemannian Hamiltonian for this structure is given by
\[
H : \mathrm{T}^*(\mathds{G}_\alpha) \to \mathds{R} : (u \diff x|_{(x, y)} + v \diff y|_{(x, y)}, x, y) \mapsto \frac{1}{2}(u^2 + v^2 x^{2 \alpha}).
\]

A normal extremal $\lambda(t) = (x(t), y(y), u(t) \diff x|_{\gamma(t)} + v(t) \diff y|_{\gamma(t)}) \in \T^*(\G_\alpha)$  is thus characterised by the Hamiltonian system \cref{eq:hamiltoneqs}
\begin{equation}
    \label{eq:ODEagrushin}
    \left\{
    \begin{aligned}
    \dot{u}(t) ={}& - \alpha v(t)^2 x(t)^{2(\alpha - 1)} x(t) \\
    \dot{v}(t) ={}& 0 \\
    \dot{x}(t) ={}& u(t) \\
    \dot{y}(t) ={}& v(t) x(t)^{2 \alpha}.
    \end{aligned}
    \right.
\end{equation}
These equations can be integrated with the help of special functions, known as generalised trigonometric functions. Let us define the $\alpha$-sine function as the solution of the following ordinary differential equation
\begin{equation}
\label{eq:sina}
\left\{
\begin{aligned}
    &f''(t) = - \alpha f(t)^{2 (\alpha - 1)} f(t)  \\
    &f(0) = 0, \ \ f'(0) = 1.
\end{aligned}
\right.
\end{equation}
It can be shown that the solution to this differential equation is periodic with period $2 \pi_\alpha$, where the constant $\pi_\alpha$ is given by
\[ 
\pi_{\alpha} := 2 \int_0^1 \dfrac{1}{\sqrt{1 - t^{2 \alpha}}} \mathrm{d}t = \mathrm{B}\left( \dfrac{1}{2}, 1 - \dfrac{1}{2 \alpha} \right).
\]
Here the function $\mathrm{B}(\cdot, \cdot)$ stands for the complete beta function. The solution to $\cref{eq:sina}$, the $\alpha$-sine, will be denoted by $\sin_\alpha$, while the $\alpha$-cosine function is defined as $\cos_\alpha := \sin_\alpha'$. It should be noted that these functions satisfy the following $\alpha$-trigonometric identity:
\begin{equation}
    \label{eq:propagrushin}
    \begin{aligned}
    & \sin_\alpha^{2 \alpha}(x) + \cos_\alpha^2(x) = 1, \text{ for all } x \in \R.
    \end{aligned}
\end{equation}

The general solution to \cref{eq:ODEagrushin} can be written in a relatively clear form using these special functions (see \cite[Theorem 15.]{borza2022}):
\begin{equation}
    \label{eq:agrushingeodesic}
    \left\{
    \begin{aligned}
    u(t) ={}& A \omega \cos_\alpha(\omega t + \phi) \\
    v(t) ={}& v_0 \\
    x(t) ={}& A \sin_\alpha(\omega t + \phi) \\
    y(t) ={}& y_0 + \frac{t \left(u_0^2 + v_0^2 x_0^{2 \alpha}\right) + u_0 x_0 -u(t) x(t)}{v_0 (\alpha+1)}
    \end{aligned}
    \right.
\end{equation}
where $A$, $\omega$, and $\phi$ are uniquely determined by the initial condition $\lambda(0) = (q, \lambda_0) = (x_0, y_0, u_0 \diff x|_{q} + v_0 \diff y|_{q}) \in \T^*(\G_\alpha)$:
\begin{equation}
\label{eq:initialconditions}
		\begin{gathered}
			A^2 \omega^2 = u_0^2 + v_0^2 x_0^{2 \alpha}, \ \ \omega^2 = v_0^2 A^{2(\alpha - 1)}, \\ 
			x_0 = A \sin_\alpha(\phi) \text{ and } u_0 = A \omega \cos_\alpha(\phi).
		\end{gathered}
\end{equation}
To examine the exponential map of $\G_\alpha$ in the neighbourhood of singularities, it is essential to study the Jacobi equation, as explained in \cref{subsec:jacobifields}. 

Since the structure is rank-varying, it is not possible to construct the Zelenko-Li frame that was mentioned in \cref{zelenkoli}. Therefore, the specific (cartesian) symplectic moving frame along $\lambda(t)$ on the cotangent bundle $\T^*(\G_\alpha)$
\begin{equation}
    \label{eq:agrushinmovingframe}
    (E_a, E_b, F_a, F_b) = (\partial_u, \partial_v, \partial_x, \partial_y)
\end{equation}
will be chosen moving forward.

With respect to this frame, the linear map 
\begin{equation*}
    \diff_{(p, \lambda_0)} \me^{t \overrightarrow{H}} : \mathrm{T}_{(p, \lambda_0)}(\mathrm{T}^*(\mathds{G}_\alpha)) \to \mathrm{T}_{\lambda(t)}(\mathrm{T}^*(\mathds{G}_\alpha))
\end{equation*}
is simply the Jacobian matrix of the map $(u_0, v_0, x_0, y_0) \mapsto (u(t), v(t), x(t), y(t))$. The partial derivatives of $x(t)$ with respect to the initial conditions $u_0, v_0, x_0, y_0$ can be computed by using the relations \cref{eq:initialconditions} and the properties of the $\alpha$-trigonometric functions, using the abuse of notation $H := H(q, \lambda_0)$:
\begin{equation}
    \label{eq:dx(t)}
    \left\{
    \begin{aligned}
    \partial_{u_0} x(t) ={}& \frac{((\alpha-1) t u_0-x_0) u(t) + u_0 x(t)}{2 H \alpha} \\
    \partial_{v_0} x(t) ={}& \frac{\left(t \left(\alpha \left(2 H-u_0^2\right) + u_0^2\right) + u_0 x_0\right) u(t) - u_0^2 x(t)}{2 H \alpha v_0} \\
    \partial_{x_0} x(t) ={}& \frac{\left((\alpha - 1) t \left(2 H - u_0^2\right) + u_0 x_0\right) u(t) +\left(2 H - u_0^2\right) x(t)}{2 H x_0} \\
    \partial_{y_0} x(t) ={}& 0
    \end{aligned}.
    \right.
\end{equation}
Because of \cref{eq:ODEagrushin}, the partial derivatives of $u(t)$ with respect to the initial conditions can be obtained by considering the time derivatives of  \cref{eq:dx(t)}.
\begin{equation}
    \label{eq:du(t)}
    \left\{
    \begin{aligned}
    \partial_{u_0} u(t) ={}& \frac{\alpha u_0 u(t) + ((\alpha - 1) t u_0 - x_0) \dot{u}(t)}{2 H \alpha} \\
    \partial_{v_0} u(t) ={}& \frac{\alpha \left(2 H-u_0^2\right) u(t) + \left(t \left(\alpha \left(2 H-u_0^2\right)+u_0^2\right)+u_0 x_0\right) \dot{u}(t)}{2 \alpha H v_0} \\
    \partial_{x_0} u(t) ={}& \frac{\alpha \left(2 H-u_0^2\right) u(t)+ \left((\alpha-1) t \left(2 H-u_0^2\right)+u_0 x_0\right)\dot{u}(t)}{2 H x_0} \\
    \partial_{y_0} u(t) ={}& 0
    \end{aligned}.
    \right.
\end{equation}
The partial derivatives of $y(t)$, which are not explicitly written here, can be obtained as a combination of \cref{eq:dx(t)} and \cref{eq:du(t)} by considering the last identity of \cref{eq:agrushingeodesic}.

The differential system \cref{eq:noncanonicalmovingframe} characterising Jacobi fields in $\mathds{G}_\alpha$ can be written with respect to the frame \cref{eq:agrushinmovingframe}. However, it is important to note again that this frame is non-canonical, as the sub-Riemannian structure is rank-varying, making it impossible for geodesics that cross the singular region to be ample and equiregular and therefore impossible to apply \cref{zelenkoli}.

\begin{proposition}
    Let $\lambda(t)$ denote the normal extremal of $\G_\alpha$ with initial condition $(q, \lambda_0) = (x_0, y_0, u_0 \diff x|_{q} + v_0 \diff y|_{q}) \in \T^*(\G_\alpha)$. If $(E_a, E_b, F_a, F_b)$ denotes the symplectic frame along $\lambda(t)$ given by $(\partial_u, \partial_v, \partial_x, \partial_y)$, then the vector field $\mathcal{J}(t) = \sum_{k \in \{a, b\}} p_k(t) E_k(t) + x_k(t) F_k(t)$ is a Jacobi field along $\lambda(t)$ if and only if
    \begin{equation}
    \label{jacobifieldsagrushin}
    \begin{pmatrix}
			\dot{p}_a(t) \\
			\dot{p}_b(t) \\
            \dot{x}_a(t) \\
			\dot{x}_b(t)
    \end{pmatrix}
    =
    \begin{pmatrix}
			-C_1(t) & -R(t) \\
			C_2(t) & C_1(t)\tran
    \end{pmatrix}
    \begin{pmatrix}
			p_a(t) \\
			p_b(t) \\
            x_a(t) \\
			x_b(t)
    \end{pmatrix},
    \end{equation}
    where
    \[
    C_1(t) = \begin{pmatrix}
			0 & 2 \alpha v_0 x(t)^{2(\alpha - 1)} x(t) \\
			0 & 0 
    \end{pmatrix}, \
    C_2(t) = \begin{pmatrix}
			1 & 0 \\
			0 & x^{2 \alpha}(t)
    \end{pmatrix},
    \]
    \[
    R(t) = \begin{pmatrix}
			\alpha (2 \alpha - 1) v_0^2 x(t)^{2(\alpha - 1)}  & 0 \\
			0 & 0
    \end{pmatrix}.
    \]
\end{proposition}

\begin{proof}
    Obtaining the result is a matter obtaining the Lie derivatives of the elements of the frame $(E_a, E_b, F_a, F_b)$ along $\lambda(t)$ in the direction of $\overrightarrow{H}$, as seen in \cref{subsec:jacobifields}, which is long and tedious computation. It can be done by noting that
\begin{equation*}
    \begin{aligned}
    A(u(t), v(t), x(t), y(t)) ={}& A(u_0, v_0, x_0, y_0) \\
    \omega(u(t), v(t), x(t), y(t)) ={}& \omega(u_0, v_0, x_0, y_0) \\
    \phi(u(t), v(t), x(t), y(t)) ={}& \omega(u_0, v_0, x_0, y_0) t + \phi(u_0, v_0, x_0, y_0)
    \end{aligned},
\end{equation*}
    where, for example, $A(u(t), v(t), x(t), y(t))$ denotes the value of $A$ for the initial conditions $u(t), v(t), x(t), y(t)$ from \cref{eq:initialconditions}. Thus, with the help of the expression of $\diff_{(p, \lambda_0)} \me^{t \overrightarrow{H}}$ in terms of \cref{eq:dx(t)} and \cref{eq:du(t)} and the properties of the functions \cref{eq:agrushingeodesic}, one can take the Lie derivative of $\partial_u$ in the direction of $\overrightarrow{H}$ as introduced in \cref{eq:liederivatire} and obtain
    \begin{align*}
        \dot{\partial}_u ={}& \frac{\diff}{\diff\epsilon}\Big|_{\epsilon=0} (\diff_{\lambda(t + \epsilon)} \me^{-\epsilon \overrightarrow{H}})[\partial_u (t + \epsilon)] \\
        ={}& \frac{\diff}{\diff\epsilon}\Big|_{\epsilon=0} \begin{pmatrix}
			\frac{\alpha u(t) u(t + \epsilon) - ((\alpha - 1) \epsilon u(t + \epsilon) + x(t + \epsilon)) u'(t)}{2 \alpha H} \\
			0 \\
            \frac{x(t) u(t + \epsilon) - ((\alpha-1) \epsilon u(t + \epsilon) + x(t + \epsilon)) u(t)}{2 \alpha H} \\
			\star
    \end{pmatrix} = \begin{pmatrix}
			0 \\
			0 \\
            -1 \\
			0
    \end{pmatrix}.
    \end{align*}
    The derivatives for the other elements of the symplectic frame are obtained in a similar fashion. The conclusion then follows with \cref{eq:jacobifields}.
\end{proof}

The explicit integration of the differential equation \cref{jacobifieldsagrushin} can be performed using $\alpha$-trigonometric functions. Writing the initial conditions of a vector field satisfying \cref{jacobifieldsagrushin} as $(p_a(0), p_b(0), x_a(0), x_b(0)) = (p_{a0}, p_{b0}, x_{a0}, x_{b0})$, we have that $p_b(t) = p_{b0}$ and the equation for $x_a(t)$ is
\[
\ddot{x}_a(t) + \alpha v_0 x(t)^{2(\alpha - 1)} \left[ 2 p_{b0}  x(t) + (2 \alpha - 1) v_0 x_a(t) \right] = 0,
\]
which can be integrated with the ansatz $x_a(t) = k_1 x(t) + (k_2 + k_3 t) u(t)$. After some computations, one can find that
\begin{align*}
    k_1 ={}& \frac{\alpha x_{a0} v_0^3 x_0^{2 \alpha}+ p_{a0} u_0 v_0 x_0 - p_{b0} u_0^2 x_0}{\alpha v_0 x_0 \left(u_0^2 + v_0^2 x_0^{2 \alpha}\right)}\\
    k_2 ={}& \frac{\alpha x_{a0} u_0 v_0 - p_{a0} v_0 x_0 + p_{b0} u_0 x_0}{\alpha v_0 \left(u_0^2 + v_0^2 x_0^{2 \alpha}\right)}\\
    k_{3} ={}& \frac{(\alpha - 1) v_0 k_1 + p_{b0}}{v_0}.
\end{align*}

The function $p_a(t)$ is obtained by differentiating $x_a(t)$, namely $p_a(t) = (k_1 + k_3) u(t) +(k_2 + k_3 t) \dot{u}(t)$. Finally, we can integrate the equation
\[
\dot{x}_b(t)= p_{b0} x(t)^{2 \alpha} + 2 \alpha v_0 x(t)^{2(\alpha - 1)} x(t) x_a(t)
\]
by noting that $\dot{y}(t) = v_0 x(t)^{2 \alpha}$, $\dot{u}(t)= - \alpha v_0^2 x(t)^{2(\alpha - 1)} x(t)$ by \cref{eq:ODEagrushin}, and that
\begin{align*}
    \int_0^t u(s) \dot{u}(s) \diff s ={} & \frac{1}{2} \left(u(t)^2-u_0^2\right); \\
    \int_0^t s u(s) \dot{u}(s) \diff s ={}& \frac{u_0 x_0 - \alpha t \left(u_0^2 + v_0^2 x_0^{2 \alpha}\right) + (\alpha + 1) t u(t)^2 - u(t) x(t)}{2 (\alpha+1)}.
\end{align*}

For the sake of completeness, the full expression for $x_b(t)$ should also be written:
\begin{align*}
    x_b(t) ={}& q_{b0} +  \left(\frac{p_{b0}}{v_0} + 2 \alpha k_1\right) \frac{t \left(u_0^2 + v_0^2 x_0^{2 \alpha}\right) + u_0 x_0 -u(t) x(t)}{v_0 (\alpha+1)} \\
    &- \frac{k_2}{v_0} \left(u(t)^2-u_0^2\right) - k_3 \frac{u_0 x_0 - \alpha t \left(u_0^2 + v_0^2 x_0^{2 \alpha}\right) + (\alpha + 1) t u(t)^2 - u(t) x(t)}{v_0 (\alpha+1)}.
\end{align*}

Let us introduce additional notation: $x_1, u_1, \dot{x}_1$ and $\dot{u}_1$ will mean $x(1), u(1), \dot{x}(1)$ and $\dot{u}(1)$ respectively. These variables therefore depend on the initial conditions $x_0, u_0, v_0$. Similarly, we will write $p_{a1}, p_{b1}, x_{a1}, x_{b1}$ for $p_a(1), p_b(1), x_a(1), x_b(1)$ when it comes to the coordinate representation of a Jacobi field at time $t = 1$.

\begin{proposition}
    \label{prop:singGa}
    The covector $\lambda_0 = u_0 \diff x|_q + v_0 \diff y|_q \in \T^*_q(\mathds{G}_\alpha)$ is a conjugate covector of $q = (x_0, y_0) \in \mathds{G}_\alpha$ with $H(\lambda_0) \neq 0$ if and only if 
    \begin{equation}
    \label{eq:conjugatelocusagrushin}
        u_1 (u_0 + x_0) - u_0 x_1 = 0 \text{ and } v_0 \neq 0.
    \end{equation}
    Furthermore, the conjugate vectors $\lambda_0 \in \mathrm{Conj}_q(\G_\alpha)$ are all of order one, and
    \begin{equation}
        \label{eq:kerneldexpagrushin}
        \Kern(\diff_{\lambda_0} \exp_q) = \vspan\left\{ v_0 x_0^{2 \alpha} \partial_u - u_0 \partial_v
    \right\}.
    \end{equation}
\end{proposition}

\begin{proof}
    By \cref{eq:kernelconjugatelocus}, we know that
    \[
    \Kern(\diff_{\lambda_0} \exp_q) = \left\{ p_{a0} E_1(0) + p_{20} E_2(0) \mid (p(t), x(t)) \text{ solves } \cref{jacobifieldsagrushin}, \ x(0) = x(1) = 0
    \right\}.
    \]
    From the explicit expression for the Jacobi fields of $\G_\alpha$, one can see that under the assumption that $x_{11} = x_{12} = 0$, and if $v_0 \neq 0$ we have
    \begin{equation}
    \label{agrushinx11x21}
    \begin{pmatrix}
			x_{11} \\
                u_1 x_{11} + v_0 x_{21}
    \end{pmatrix}
    =
    \begin{pmatrix}
			\frac{u_0 x_1 + u_1 ((\alpha - 1) u_0 - x_0)}{\alpha (u_0^2 + v_0^2 x_0^{2 \alpha})} & \tfrac{\alpha u_1 v_0^2 x_0^{2 \alpha} - u_0^2 x_1 + u_0^2 u_1 + u_0 u_1 x_0}{\alpha v_0 (u_0^2 + v_0^2 x_0^{2 \alpha})} \\
               u_0 & v_0 x_0^{2 \alpha}
    \end{pmatrix}
    \begin{pmatrix}
			p_{10} \\
                p_{20}
    \end{pmatrix}.
    \end{equation}
    
    \textbf{First case:} $v_0 = 0$.
    By continuity of the geodesic flow with respect to the initial conditions, this is equivalent to considering \cref{agrushinx11x21} when $v_0$ tends to $0$. This yields $x_{a1} = p_{b0}$, $x_{a1} = p_{b0}$, and there are thus no conjugate covector along such a geodesic. From now on therefore, we will assume that assume that $v_0 \neq 0$.
    
    \textbf{Second case:} $u_0 = 0$ (in particular $x_0 \neq 0$). The matrix appearing in \cref{agrushinx11x21} then becomes
    \[
    \begin{pmatrix}
			-\frac{u_1}{\alpha v_0^2 x_0^{2(\alpha - 1)} x_0} & \tfrac{u_1}{v_0} \\
               0 & v_0 x_0^{2 \alpha}
    \end{pmatrix}.
    \]
    It has a non-trivial kernel if and only if $u_1 = 0$, or equivalently $\sin_\alpha(\omega) = 0$ and $\cos_\alpha(\phi) = 0$, which means that in that case there exists a non-trivial Jacobi field if and only if $p_{b0} = 0$ and $p_{a0}$ is arbitrary, and then
    \[
    \Kern(\diff_{\lambda_0} \exp_q) = \vspan \left\{E_a(0) \right\}.
    \]

    \textbf{Third case:} $u_0 \neq 0$ and $u_0 x_1 + u_1 ((\alpha - 1) u_0 - x_0) = 0$. Since $H(\lambda(t)) \neq 0$ is constant, we have in particular that $u_1 \neq 0$. 
    
    The matrix in \cref{agrushinx11x21} reduces to
    \[
    \begin{pmatrix}
			0 & \tfrac{u_1}{v_0} \\
               u_1 & v_0 x_0^{2 \alpha}
    \end{pmatrix},
    \]
    which has a trivial kernel. Thus, this case does not produce any conjugate covector.

    \textbf{Fourth case:} $u_0 \neq 0$, $u_0 x_1 + u_1 ((\alpha - 1) u_0 - x_0) \neq 0$ and $x_0 = 0$. These conditions imply that $u_1 \neq 0$. This time, the matrix \cref{agrushinx11x21} simplifies to
    \[
    \begin{pmatrix}
			\frac{x_1 + (\alpha - 1) u_1}{\alpha u_0} & \tfrac{u_1 - x_1}{v_0} \\
               u_1 & 0
    \end{pmatrix}
    \]
    and there is a non-trivial kernel if and only if $u_1 = x_1$, and thus $x_1 \neq 0$. The corresponding non-trivial Jacobi fields along $\lambda(t)$ have initial conditions $p_{a0} = 0$ and arbitrary $p_{b0}$, and 
    \[
    \Kern(\diff_{\lambda_0} \exp_q) = \vspan \left\{E_b(0) \right\}.
    \] 

    \textbf{Fifth case:} $u_0 \neq 0$, $u_0 x_1 + u_1 ((\alpha - 1) u_0 - x_0) \neq 0$ and $x_0 \neq 0$. Here the matrix in \cref{agrushinx11x21} is similar to
    \[
    \begin{pmatrix}
			1 & \star \\
               0 & -\left(\frac{u_0^2 + v_0^2 x_0^{2 \alpha}}{v_0}\right)\frac{u_1 (u_0 + x_0) - u_0 x_1}{u_0 x_1 + u_1 ((\alpha - 1)u_0  - x_0)}
    \end{pmatrix}
    \]
    which yields a non-trivial kernel if and only if $u_1 (u_0 + x_0) - u_0 x_1 = 0$. In that case, that same matrix becomes
    \[
    \begin{pmatrix}
			1 & \frac{v_0 x_0^{2 \alpha}}{u_0} \\
               0 & 0
    \end{pmatrix}.
    \]
    The corresponding Jacobi fields have therefore initial conditions $p_{b0} \in \mathds{R}$ and $p_{a0} = -\frac{v_0 x_0^{2 \alpha}}{u_0} p_{b0}$, in which case the expression of the kernel is
    \[
    \Kern(\diff_{\lambda_0} \exp_q) = \vspan\left\{ v_0 x_0^{2 \alpha} \partial_u - u_0 \partial_v
    \right\}.
    \]
\end{proof}
The next step in our study of the singularities of the exponential map of $\G_\alpha$ is to prove that the equation \cref{eq:conjugatelocusagrushin} defines a submanifold of $\T^*_{(x_0, y_0)}(\G_\alpha)$ through the preimage theorem. In the language of Whitney, this means that conjugate covectors of $\exp_q$ are \textit{good singularities} (see \cref{def:goodsingularity}).

\begin{proposition}
    \label{prop:conjlocusGa}
    For all $q \in \mathds{G}_\alpha$, the conjugate locus $\mathrm{Conj}_q(\mathds{G}_\alpha)$ is a submanifold of $\mathrm{T}^*_q(\mathds{G}_\alpha)$ of codimension 1.
\end{proposition}

\begin{proof}
    For a given $x_0 \in \mathds{R}$, we are going to show that $0$ is a regular value of the map
    \[
    f_{x_0}(u_0, v_0) = u_1(u_0 + x_0) - u_0 x_1.
    \]
    Recall that $x_1$ and $u_1$ are functions of $x_0, u_0$ and $v_0$. The derivatives of $u_1$ and $x_1$ with respect to $u_0$ or $v_0$ coincide with the expressions in \cref{eq:dx(t)} and \cref{eq:du(t)}, taking $t = 1$. 
    
    Let $q = (x_0, y_0) \in \G_\alpha$ and $\lambda_0 = u_0 \diff x|_{q} + v_0 \diff y|_{q} \in \mathrm{Conj}_q(\mathds{G}_\alpha)$ with $H(q, \lambda_0) \neq 0$, that is to say $u_0$ and $v_0$ are such that the condition \cref{eq:conjugatelocusagrushin} holds. We find that
    \begin{equation}
    \label{eq:Tconjuagrushin}
    \partial_{u_0} f_{x_0}(u_0, v_0) =
    \begin{cases}
      \frac{u_1 v_0^2 x_0^{2 a}}{u_0^2 + v_0^2 x_0^{2 \alpha}} & \text{if $u_0 + x_0 = 0$}\\
      \frac{\dot{u}_1 (u_0 + x_0)^2 ((\alpha - 1) u_0-x_0) - \alpha v_0^2 x_1 x_0^{2 \alpha} x_0}{\alpha (u_0 + x_0) \left(u_0^2 + v_0^2 x_0^{2 \alpha}\right)} & \text{otherwise}
    \end{cases},    
\end{equation}
and
\begin{equation}
    \label{eq:Tconjvagrushin}
    \partial_{v_0} f_{x_0}(u_0, v_0) =
    \begin{cases}
      \frac{u_1 v_0 x_0^{2 \alpha} x_0}{u_0^2 + v_0^2 x_0^{2 \alpha}} & \text{if $u_0 + x_0 = 0$}\\
      \frac{\dot{u}_1 (u_0 + x_0)^2 \left(u_0^2 + \alpha v_0^2 x_0^{2 \alpha} + u_0 x_0\right) + \alpha u_0 v_0^2 x_1 x_0^{2 \alpha} x_0}{\alpha v_0 (u_0 + x_0) \left(u_0^2 + v_0^2 x_0^{2 \alpha}\right)} & \text{otherwise}
    \end{cases}.  
\end{equation}

If $u_0 + x_0 = 0$, then $\partial_{v_0} f_{x_0}(u_0, v_0) = \partial_{u_0} f_{x_0}(u_0, v_0) = 0$ would actually imply that $H = 0$, contradicting our hypothesis. On the other hand, if $u_0 + x_0 \neq 0$, then
\[
\frac{u_0}{v_0} \partial_{u_0} f_{x_0}(u_0, v_0) + f_{u_0}(u_0, v_0) = \frac{\dot{u}_1 (u_0 + v_0)}{v_0} = - \alpha x_1^{2 (\alpha - 1)} x_1 v_0 (u_0 + v_0),
\]
which vanishes if and only $x_1 = 0$ and by \cref{eq:conjugatelocusagrushin} this again implies that $H = 0$.

In particular, the linear map $\diff f_{x_0} (u_0, v_0)$ is surjective and the result follows.
\end{proof}

As explained in \cref{subsec:whitneyfoldpleat}, the nature of the exponential map near it singularities will be determined by how $\Kern(\diff_{\lambda_0} \exp_q)$ and $\T_{\lambda_0}(\mathrm{Conj}_q(\G_\alpha))$ interact at $\lambda_0 \in \mathrm{Conj}_q(\G_\alpha)$.

\begin{theorem}
\label{theo:normalformsGa}
On a dense non-empty open subset $C^0$ of $\mathrm{Conj}_q(\G_\alpha)$, the exponential map $\exp_q : \T^*(\G_\alpha) \to \G_\alpha$ is equivalent to $f : \R^2 \to \R^2 : (x, y) \mapsto (x^2, y)$ in the neighbourhood of any $\lambda_0 \in C^0$.
\end{theorem}

\begin{proof}
    By \cref{prop:conjlocusGa}, we can choose a one-dimensional open connected submanifold $C$ of $\mathrm{Conj}_q(\G_\alpha)$ in $\T^*_q(\G_\alpha)$ containing $\lambda_0$. In particular, the conjugate vectors in $C$ have all order 1. We thus write $C^0$ (resp. $C^1$) for the set of covectors $\overline{\lambda}_0 \in C$ such that 
\[
\dim \bigl[ \Kern(\diff_{\overline{\lambda}_0} \exp_q) \cap \T_{\overline{\lambda}_0}(\mathrm{Conj}_q(\G_\alpha)) \bigr] = 0 \text{ (resp. = 1)}.
\]
The condition $\lambda_0 \in C^1$ can be expressed as a system of two equations
\begin{equation*}
    \left\{
    \begin{aligned}
    u_1 (u_0 + x_0) - u_0 x_1 ={}& 0 \\
    \det 
    \begin{pmatrix}
			\partial_{u_0} f_{x_0}(u_0, v_0) & v_0 x_0^{2 \alpha} \\
            \partial_{v_0} f_{x_0}(u_0, v_0) & -u_0
    \end{pmatrix}
    ={}& 0 
    \end{aligned}.
    \right.
\end{equation*}
The set $C^1$ is thus closed and nowhere dense by analycity. Thus $C^0$ is dense in $C^0 \cup C^1$.

The condition $\lambda_0 \in C^0$ corresponds exactly to $\lambda_0$ being in that case a \textit{fold singularity} in the sense of \cref{def:foldsingularity}, and the normal form of a Whitney fold follows from \cref{theo:whitneyfoldpleat}.
\end{proof}

From \cref{theo:normalformsGa} readily follows the following fact.

\begin{proposition}
    \label{prop:localnoninjGa}
    The exponential map $\exp_p : \T^*_q(\G_\alpha) \to M$ fails to be injective in any neighbourhood of conjugate covector $\lambda_0 \in \mathrm{Conj}_q(\G_\alpha)$.
\end{proposition}

\begin{proof}
    The normal form of $\exp_q$ around $\lambda_0$, as stated in \cref{theo:normalformsGa}, shows that $\exp_q$ is not a one-to-one function in any area near $\lambda_0$ when $\lambda_0$ is in $C^0$. Since $C^0$ is dense in $\mathrm{Conj_q}(\G_\alpha)$, the same conclusion holds for $\lambda_0$ in $C^1$.
\end{proof}

\section{The special unitary group \texorpdfstring{$\mathrm{SU}(2)$}{SU(2)}}

\label{sec:su2}

The Lie group $\mathrm{SU}(2)$ is the group of $2 \times 2$ complex matrices defined as
\[
\mathrm{SU}(2) = \left\{ 
\begin{pmatrix}
\alpha & \beta \\
- \overline{\beta} & \overline{\alpha}
\end{pmatrix} \in \mathds{C}^{2 \times 2}
\mid \alpha, \beta \in \mathds{C}, |\alpha|^2 + |\beta|^2 = 1
\right\}.
\]
An element of $\mathrm{SU}(2)$ can be identified with a pair $(\alpha, \beta) \in \mathds{C}^2$ such that $|\alpha|^2 + |\beta|^2 = 1$. With this identification, the group structure is given by
\[
(\alpha, \beta) \star (\alpha', \beta') = (\alpha \alpha' - \beta \overline{\beta'}, \alpha \beta' + \overline{\alpha'} \beta).
\]
The identity element is the identity matrix which corresponds to $(1, 0)$ and the inverse is $(\alpha, \beta)^{-1} = (\overline{\alpha}, - \beta)$.

The Lie algebra $\mathfrak{su}(2)$ of $\mathrm{SU}(2)$ is naturally identified with the space of antihermitian traceless $2 \times 2$ complex matrices:
\[
\mathfrak{su}(2) = \left\{ 
\begin{pmatrix}
\I a & b \\
- \overline{b} & -\I a
\end{pmatrix} \in \mathds{C}^{2 \times 2}
\mid a \in \R, b \in \C
\right\} \cong \mathrm{T}_e(\mathrm{SU}(2)).
\]

A basis for $\mathfrak{su}(2)$ is given by $X_0, X_1, X_2$ where
\[
X_0 := 
\frac{1}{2}
\begin{pmatrix}
\I & 0 \\
0 & -\I
\end{pmatrix}, \ 
X_1 := \frac{1}{2}
\begin{pmatrix}
0 & 1 \\
- 1 & 0
\end{pmatrix}\text{, and }
X_2 :=
\frac{1}{2} 
\begin{pmatrix}
0 & \I \\
\I & 0
\end{pmatrix}.
\]
We also have the commutation relations 
\begin{equation}
    \label{eq:bracketrelationsSU2}
    [X_1, X_2] = X_0, \ [X_2, X_0]= X_1, \text{ and } [X_0, X_1] = X_2.
\end{equation} 

We construct a sub-Riemannian structure in the following way: let $\langle \cdot, \cdot \rangle$ the inner product on $\mathfrak{su}(2)$ that turns $X_1, X_2$ and $X_0$ into an orthonormal basis. We also define $\mathbf{d} := \mathrm{span}\left\{X_1, X_2\right\}$ and $\mathbf{s} = \mathrm{span}\left\{X_0\right\}$, and we denote by $\langle \cdot, \cdot \rangle_{\mathbf{d}}$ the restriction of $\langle \cdot, \cdot \rangle$ to $\mathbf{d}$. By left-translating $\mathbf{d}$ and $\langle \cdot, \cdot \rangle_{\mathbf{d}}$, we obtain a well-defined sub-Riemannian structure on $\mathrm{SU}(2)$. This kind of sub-Riemannian structure on a Lie group is often called a $\mathbf{d} \oplus \mathbf{s}$ sub-Riemannian structure and the expression of the normal geodesics for such a structure are well-known (see \cite[Section 7.7.1]{comprehensive2020}). The metric $\langle \cdot, \cdot \rangle$ is used to identify vectors and covectors.

The sub-Riemannian Hamiltonian of this structure is given by
\[
H : \mathrm{T}^*(\mathrm{SU}(2)) \to \mathds{R} : (u X_1(\alpha, \beta) + v X_2(\alpha, \beta) + w X_0(\alpha, \beta), \alpha, \beta) \mapsto \frac{1}{2}(u^2 + v^2).
\]
By left-invariance, it is enough to write the geodesics starting from the identity. In this case, the normal extremal starting from $(q, \lambda_0) \in \T^*(\mathrm{SU}(2))$ with $q = 0$ and $\lambda_0 = u_0 X_1 + v_0 X_2 + w_0 X_0$ is
\begin{equation}
    \label{eq:SU2geodesic}
    \left\{
    \begin{aligned}
    \alpha(t) ={}& \left(\cos \left(\frac{w_0 t}{2}\right)- \I \sin \left(\frac{w_0 t}{2}\right)\right) \left(\cos \left(\frac{|\lambda_0| t}{2} \right)+ \I \frac{w_0}{|\lambda_0|} \sin \left(\frac{|\lambda_0| t}{2}\right)\right) \\
    \beta(t) ={}& \frac{1}{|\lambda_0|} (u_0 + \I v_0) \sin \left(\frac{|\lambda_0| t}{2}\right) \left(\cos \left(\frac{w_0 t}{2}\right)+ \I \sin \left(\frac{w_0 t}{2}\right)\right) \\
    u(t) ={}& u_0 \cos(w_0 t) - v_0 w_0 \sin(w_0 t)  \\
    v(t) ={}& v_0 \cos(w_0 t) + u_0 w_0 \sin(w_0 t) \\
    w(t) ={}& w_0
    \end{aligned},
    \right.
\end{equation}
where $|\lambda_0|^2 = \langle \lambda_0, \lambda_0 \rangle = u_0^2 + v_0^2 + w_0^2$.

This structure admits the kind of normal frame that we described in \cref{zelenkoli}. To introduce it, we must first establish some notation, closely following \cite[Section 7.5.]{curvature}. For $i \in \{0, 1, 2\}$, we define the functions $h_i : \T^*(\mathrm{SU}(2)) \to \R$ by $h_i(\lambda) := \langle \lambda, X_i(\pi(\lambda)) \rangle$. Then, the vector fields
\[
\overrightarrow{h_0}, \overrightarrow{h_1}, \overrightarrow{h_2}, \partial_{h_0}, \partial_{h_1}, \partial_{h_2}
\]
form a local frame of vector fields on $\T^*(\mathrm{SU}(2))$. We can also introduce cylindrical coordinates $\theta, \rho : \T^*(\mathrm{SU}(2)) \to \R$ as $h_1 = \rho \cos \theta$, $h_2 = \rho \sin \theta$ and use the local frame
\[
\overrightarrow{h_0}, \overrightarrow{h_1}, \overrightarrow{h_2}, \partial_{h_0}, \partial_{\rho}, \partial_{\theta}.
\]
The Euler vector field is the generator of the dilations $\lambda \mapsto c \lambda$ on the fibers of $\T^*(\mathrm{SU}(2))$, in coordinates
\[
\mathfrak{e} = h_0 \partial_{h_0} + h_1 \partial_{h_1} + h_2 \partial_{h_2} = \rho \partial_\rho + h_0 \partial_{h_0}.
\]
On the other hand, the Hamiltonian is written as $H = \frac{1}{2} (h_1^2 + h_2^2)$ while its symplectic gradient is $\overrightarrow{H} = h_1 \overrightarrow{h_1} + h_2 \overrightarrow{h_2}$. We finally define $\overrightarrow{H'} = [\partial_\theta, \overrightarrow{H}]$. It is not difficult to see that thanks to the bracket relations \cref{eq:bracketrelationsSU2}, we have
\begin{align*}
\overrightarrow{h_0} ={}& \tilde{X_0} + h_2 \partial_{h_1} - h_1 \partial_{h_2} = \Tilde{X_0} - \partial_\theta \\
\overrightarrow{h_1} ={}& \tilde{X_1} -  h_2 \partial_{h_0} + h_0 \partial_{h_2} \\
\overrightarrow{h_2} ={}& \tilde{X_2}  - h_0 \partial_{h_1} + h_1 \partial_{h_0}
\end{align*}
where $\tilde{X_i}(\lambda) \in \mathrm{T}_{\lambda}(\mathrm{T}^*(M))$ are chosen such that $\diff_\lambda \pi[\tilde{X_i}(\lambda)] = X_i(\pi(\lambda))$. This allows us to compute
\begin{equation}
    \label{eq:HH'}
    \begin{aligned}
     \overrightarrow{H} ={}& h_1 \tilde{X_1} + h_2 \tilde{X_2} + h_0 (h_1 \partial_{h_2} - h_2 \partial_{h_1}) \\
     \overrightarrow{H}' ={}& h_1 \overrightarrow{h_2} - h_2 \overrightarrow{h_1} - 2 H \partial_{h_0} + h_0 (h_1 \partial_{h_1} + h_2 \partial_{h_2}) = h_1 \tilde{X_2} - h_2  \tilde{X_1} \\
     [\overrightarrow{H}, \overrightarrow{H}'] ={}& 2 H \overrightarrow{h_0} - h_0 \overrightarrow{H} + |\lambda_0|^2 (h_1 \partial_{h_2} - h_2 \partial_{h_1}) = 2 H \tilde{X_0} - h_0 (h_1 \tilde{X_1} + h_2 \tilde{X_2})
    \end{aligned}.
\end{equation}

Geodesics in 3D contact sub-Riemannian structures, including $\mathrm{SU}(2)$, are ample and equiregular, allowing for the use of a canonical symplectic moving frame as stated in \cref{zelenkoli}. The construction of such a frame for any 3D contact structure can be found in \cite[Section 7.5]{curvature} and we apply it here to $\mathrm{SU}(2)$.

\begin{proposition}
    \label{prop:zelenkoliSU2}
    Let $\lambda(t)$ be the normal extremal in $\T^*(\mathrm{SU}(2))$ starting from $(q, \lambda_0) \in \mathrm{T}^*(\mathrm{SU}(2))$. The symplectic moving frame $(E_a, E_b, E_c, F_a, F_b, F_c)$ where
    \begin{equation}
    \label{eq:SU2movingframe}
    \begin{aligned}
    E_a(t) ={}& \frac{1}{\sqrt{2 H}} \partial_\theta & F_a(t) ={}& \frac{1}{\sqrt{2 H}} \overrightarrow{H'}\\
    E_b(t) ={}& \frac{1}{\sqrt{2 H}} \mathfrak{e} & F_b(t) ={}& \frac{1}{\sqrt{2 H}} \overrightarrow{H}\\
    E_c(t) ={}& - \frac{1}{\sqrt{2 H}} \partial_{h_0} & F_c(t) ={}& \frac{1}{\sqrt{2 H}} ([\overrightarrow{H'}, \overrightarrow{H}] + |\lambda_0|^2 \partial_\theta)\\
    \end{aligned}
\end{equation}
is a canonical moving frame along $\lambda(t)$. More specifically, it satisfies the structural equations
\begin{equation}
    \label{eq:SU2structuraleqODE}
    \begin{aligned}
    \dot{E}_a(t) ={}& - F_a(t) & \dot{F}_a(t) ={}& |\lambda_0|^2 E_a(t)\\
    \dot{E}_b(t) ={}& - F_b(t) & \dot{F}_b(t) ={}& 0 \\
    \dot{E}_c(t) ={}& E_a(t) & \dot{F}_c(t) ={}& 0\\
    \end{aligned}.
\end{equation}
\end{proposition}

\begin{proof}
The existence of such a frame is a consequence of \cref{zelenkoli}. The proof of this theorem is in fact constructive and can be found in \cite{zelenkoli2009}. The algorithm for this theorem has been applied to 3D contact structures in \cite[Section 7.5]{curvature}.
The remaining task, as stated in \cite[Proposition 7.13.]{curvature}, is to determine the values of the matrix $R(t)$ mentioned in \cref{zelenkoli}, specifically, the entries $R_{aa}(t)$ and $R_{cc}(t)$, which are the only potentially non-vanishing entries. They are computed with the help of \cref{eq:HH'} as follows:
\begin{align*}
    R_{aa}(t) ={}& \frac{1}{2 H} \sigma_{\lambda(t)}([\overrightarrow{H}, \overrightarrow{H}'], \overrightarrow{H}') = 2 H + h_0^2 = |\lambda_0|^2, \text{ and } \\
    R_{cc}(t) ={}& \frac{1}{2 H} \sigma_{\lambda(t)}([\overrightarrow{H}, [\overrightarrow{H}, \overrightarrow{H}']], [\overrightarrow{H}, \overrightarrow{H}']) - \frac{1}{(2 H)^2} \sigma_{\lambda(t)}([\overrightarrow{H}, \overrightarrow{H}'], \overrightarrow{H}')^2 = 0.
\end{align*}
\end{proof}

The use of such a moving frame makes studying the conjugate locus of $\mathrm{SU}(2)$ much simpler. We start by writing and solving the Jacobi equation \cref{eq:jacobifields} for $\mathrm{SU}(2)$.

\begin{proposition}
    If $(E, F)$ denotes the symplectic frame of \cref{prop:zelenkoliSU2}, then the vector field $\mathcal{J}(t) = \sum_{k \in \{a, b, c\}} p_k(t) E_k(t) + x_k(t) F_k(t)$ along $\lambda(t)$ is a Jacobi field if and only if
    \begin{equation}
    \label{jacobifieldssu2}
    \begin{pmatrix}
			\dot{p}(t) \\
			\dot{x}(t)
    \end{pmatrix}
    =
    \begin{pmatrix}
			-C_1(t) & -R(t) \\
			C_2(t) & C_1(t)\tran
    \end{pmatrix}
    \begin{pmatrix}
			p(t) \\
			x(t)
    \end{pmatrix},
    \end{equation}
    where
    \[
    C_1(t) = \begin{pmatrix}
			0 & 0 & 1 \\
			0 & 0 & 0 \\
            0 & 0 & 0
    \end{pmatrix}, \
    C_2(t) = \begin{pmatrix}
			1 & 0 & 0 \\
			0 & 1 & 0 \\
            0 & 0 & 0
    \end{pmatrix},
    \]
    \[
    R(t) = \begin{pmatrix}
			|\lambda_0|^2 & 0 & 0 \\
			0 & 0 & 0 \\
            0 & 0 & 0
    \end{pmatrix}.
    \]
\end{proposition}

\begin{proof}
    This is a simple consequence of \cref{prop:zelenkoliSU2} with \cref{eq:jacobifields}.
\end{proof}

The equations \cref{jacobifieldssu2} can easily be integrated. We write out the solutions explicitly in order to proceed with the next steps.
\begin{equation}
    \label{eq:jacobifieldssu2explicit}
    \left\{
    \begin{aligned}
    p_a(t) ={}& p_{a0} \cos (|\lambda_0| t) - \frac{\left(|\lambda_0|^2 x_{a0} + p_{c0} \right) \sin (|\lambda_0| t)}{|\lambda_0|}\\
    p_b(t) ={}& p_{b0}\\
    p_c(t) ={}& p_{c0} \\
    x_a(t) ={}& \frac{\left(|\lambda_0|^2 x_{a0} + p_{c0} \right) \cos (|\lambda_0| t) + |\lambda_0| p_{a0} \sin (|\lambda_0| t) - p_{c0}}{|\lambda_0|^2}\\
    x_b(t) ={}& p_{b0} t + x_{b0}\\
    x_c(t) ={}& \frac{|\lambda_0| \left(|\lambda_0|^2 x_{c0} - p_{a0} \cos (|\lambda_0| t) + p_{a0} - p_{c0} t\right)+\left(|\lambda_0|^2 x_{a0} + p_{c0} \right) \sin (|\lambda_0| t)}{|\lambda_0|^3}
    \end{aligned}
    \right.
\end{equation}

\begin{proposition}
    \label{prop:conjlocuskernelSU2}
    The covector $\lambda_0 \in \T^*_q(\mathrm{SU}(2))$ is a conjugate covector of $q \in \mathrm{SU}(2)$ with $H(q, \lambda_0) \neq 0$ if and only if 
    \[
    \sin \left(\frac{|\lambda_0|}{2}\right) \left(|\lambda_0| \cos \left(\frac{|\lambda_0|}{2}\right)-2 \sin \left(\frac{|\lambda_0|}{2}\right)\right) = 0, \text{ and } |\lambda_0| \neq 0.
    \]
    Furthermore, the conjugate covectors $\lambda_0 \in \mathrm{Conj}_q(\mathrm{SU}(2))$ are all of order one, and
    \begin{equation}
    \label{eq:kernelexpSU2}
    \Kern(\diff_{\lambda_0} \exp_q)=  \vspan \left\{ |\lambda_0| \cos \left(\frac{|\lambda_0|}{2}\right) (u_0 \partial_v - v_0 \partial_u) + 4 \sin \left(\frac{|\lambda_0|}{2}\right)  \partial_w \right\}.
    \end{equation}
\end{proposition}

\begin{proof}
From the equation \cref{eq:jacobifieldssu2explicit} for the Jacobi fields of $\mathrm{SU}(2)$ and the relationship between Jacobi fields and conjugate points of \cref{prop:jacobivsconjpt}, we see that we have a conjugate covector if and only if there are $(p_{a0}, p_{b0}, p_{c0}) \neq 0$ such that
\begin{equation*}
    \underbrace{
    \begin{pmatrix}
			\frac{\sin (|\lambda_0|)}{|\lambda_0|} & 0 & \frac{\cos (|\lambda_0|)-1}{|\lambda_0|^2} \\
			0 & 1 & 0 \\
			\frac{1-\cos (|\lambda_0|)}{|\lambda_0|^2} & 0 & \frac{\sin (|\lambda_0|)-|\lambda_0|}{|\lambda_0|^3} \\
    \end{pmatrix}
    }_{=: M_{|\lambda_0|}}
    \begin{pmatrix}
			p_{a0} \\
			p_{b0} \\
			p_{c0}
    \end{pmatrix}
    =
    \begin{pmatrix}
			0 \\
			0 \\
			0
    \end{pmatrix}.
    \end{equation*}
    In that case, we will be able to compute the kernel according to \cref{prop:kernelexp}:
    \[
    \Kern(\diff_{\lambda_0} \exp_q) = \left\{ p_{a0} E_{a}(0) + p_{b0} E_{b}(0) + p_{c0} E_c(0) \mid (p_{a0}, p_{b0}, p_{c0}) \in \mathrm{Ker} (M_{|\lambda_0|}) \right\}.
    \]
    
    \textbf{First case:} $|\lambda_0| = 0$. As $|\lambda_0|$ tends to zero, the matrix $M_{|\lambda_0|}$ becomes
    \[
    \begin{pmatrix}
			1 & 0 & -\frac{1}{2} \\
			0 & 1 & 0 \\
			\frac{1}{2}& 0 & -\frac{1}{6} \\
    \end{pmatrix}
    \]
    which has a zero kernel. Therefore, when $|\lambda_0| = 0$ there are no conjugate points along the geodesic. From this point forward, we will assume that $|\lambda_0| \neq 0$
    
    \textbf{Second case:} $\sin (|\lambda_0|) = 0$ and $\cos (|\lambda_0|) = 1$. The matrix $M_{|\lambda_0|}$ is then equal to
    \[
    \begin{pmatrix}
			0 & 0 & 0 \\
			0 & 1 & 0 \\
			0 & 0 & \frac{-1}{|\lambda_0|^2}
    \end{pmatrix}
    \]
    and therefore its kernel is made of vectors of the type $(p_{a0}, 0, 0)$. Thus,
    \[
    \Kern(\diff_{\lambda_0} \exp_q) = \vspan\left\{ E_{a}(0) \right\},
    \]
    and conjugate covectors have order 1.
    
    \textbf{Third case:} $\sin (|\lambda_0|) = 0$ and $\cos (|\lambda_0|) = -1$. Here, the matrix $M_{|\lambda_0|}$ is equal to
    \[
    \begin{pmatrix}
			0 & 0 & - \frac{2}{|\lambda_0|^2} \\
			0 & 1 & 0 \\
			\frac{2}{|\lambda_0|^2} & 0 & \frac{-1}{|\lambda_0|^2}
    \end{pmatrix}
    \]
    whose determinant is $\frac{4}{|\lambda_0|^4}$ so that there is no conjugate vector.
    
    \textbf{Fourth case:} $\sin (|\lambda_0|) \neq 0$. Then, the matrix $M_{|\lambda_0|}$ is similar to
    \[
    \begin{pmatrix}
			1 & 0 & -\frac{\tan \left(\frac{|\lambda_0|}{2}\right)}{|\lambda_0|} \\
			0 & 1 & 0 \\
			0 & 0 & -\frac{\sec \left(\frac{|\lambda_0|}{2}\right) \left(|\lambda_0| \cos \left(\frac{|\lambda_0|}{2}\right)-2 \sin \left(\frac{|\lambda_0|}{2}\right)\right)}{|\lambda_0|^3}
    \end{pmatrix}
    \]
    In particular, we can not have $\sec \left(\frac{|\lambda_0|}{2}\right) = 0$ and therefore we have a conjugate point if and only if $|\lambda_0| \cos \left(\frac{|\lambda_0|}{2}\right)-2 \sin \left(\frac{|\lambda_0|}{2}\right) = 0$. In that case, the matrix then becomes
    \[
    \begin{pmatrix}
			1 & 0 & -\frac{\tan \left(\frac{|\lambda_0|}{2}\right)}{|\lambda_0|} \\
			0 & 1 & 0 \\
			0 & 0 & 0
    \end{pmatrix}
    \]
    and we then have a vanishing Jacobi fields if and only if $p_{b0} = 0$ and $p_{a0} = \frac{p_{c0}}{|\lambda_0|} \tan \left(\frac{|\lambda_0|}{2}\right)$. After some algebraic manipulations, we find
    \[
    \Kern(\diff_{\lambda_0} \exp_q) = \vspan \left\{ |\lambda_0| \cos \left(\frac{|\lambda_0|}{2}\right) E_{a}(0) - 4 \sin \left(\frac{|\lambda_0|}{2}\right)  E_{c}(0)\right\}.
    \]
\end{proof}

\begin{proposition}
\label{prop:conjSU2submanifold}
For all $q \in \mathrm{SU}(2)$, the conjugate locus $\mathrm{Conj}_q(\mathrm{SU}(2))$ is a submanifold of $\T^*_q(\mathrm{SU}(2))$ of codimension 1.
\end{proposition}

\begin{proof}
The conjugate locus is a disjoint union of two sets: $\mathrm{Conj}_q(\mathrm{SU}(2)) = C^0 \cup C^1$ where
\[
C^0 := \left\{\lambda_0 \in \T^*_q(\mathrm{SU}(2)) \mid |\lambda_0| \cos \left(\frac{|\lambda_0|}{2}\right)-2 \sin \left(\frac{|\lambda_0|}{2}\right) = 0 \right\}
\]
and
\[
C^1 := \left\{\lambda_0 \in \T^*_q(\mathrm{SU}(2)) \mid \sin\left(\frac{|\lambda_0|}{2}\right) = 0 \right\}.
\]
As in \cref{prop:conjlocusGa}, we use the implicit function theorem, showing that 0 is a regular value of $f_0(\lambda_0) =  |\lambda_0| \cos \left(\frac{|\lambda_0|}{2}\right)-2 \sin \left(\frac{|\lambda_0|}{2}\right)$ (resp. of $f_1(\lambda_0) = \sin\left(\frac{|\lambda_0|}{2}\right)$) for $C^0$ (resp. $C^1$). We find that
\begin{equation}
    \label{eq:tangentconjlocusSU2C0}
    \diff_{\lambda_0} f_0 = - \frac{1}{2} \sin\left(\frac{|\lambda_0|}{2}\right) (u_0 \diff u|_{\lambda_0} + v_0 \diff v|_{\lambda_0} + w_0 \diff w|_{\lambda_0}) \neq 0 \text{ in } C^0
\end{equation}
and
\[
\diff_{\lambda_0} f_1 = \frac{1}{2 |\lambda_0|} \cos\left(\frac{|\lambda_0|}{2}\right) (u_0 \diff u|_{\lambda_0} + v_0 \diff v|_{\lambda_0} + w_0 \diff w|_{\lambda_0}) \neq 0 \text{ in } C^1.
\]
\end{proof}

By using singularity theory, the normal form of the exponential map of $\mathrm{SU}(2)$ can be expressed near its critical points as we have previously outlined.

\begin{theorem}
    \label{prop:normalformsSU2}
    The exponential map $\exp_q : \T^*(\mathrm{SU}(2)) \to \mathrm{SU}(2)$ in the neighbourhood of $\lambda_0 \in \mathrm{Conj}_q(\mathrm{SU}(2))$ is equivalent to $f : \R^3 \to \R^3 : (x, y, z) \to (x^2, y, z)$ for $\lambda_0$ in a non-empty dense subset of $C^0$ and to $f : \R^3 \to \R^3 : (x, y, z) \to (x z, y, z)$ when $\lambda_0 \in C^1$.
\end{theorem}

\begin{proof}
    \cref{prop:conjSU2submanifold} establishes that every conjugate covector of $\mathrm{SU}(2)$ is a \textit{good singularity} in the sense of Whitney.

    When $\lambda_0 \in C^0$, we observe with the explicit expressions of \cref{eq:kernelexpSU2} and \cref{eq:tangentconjlocusSU2C0} that
    \[
    \dim \bigl[ \Kern(\diff_{\lambda_0} \exp_q) \cap \T_{\lambda_0}(\mathrm{Conj}_q(\mathrm{SU}(2)) \bigr] = 0
    \]
    if $w_0 \neq 0$. Indeed, we have
    \[
    \diff_{\lambda_0}f_0 \left[|\lambda_0| \cos \left(\frac{|\lambda_0|}{2}\right) (u_0 \partial_v - v_0 \partial_u) + 4 \sin \left(\frac{|\lambda_0|}{2}\right) \partial_w\right] = -2 w_0 |\lambda_0| \sin^2 \left(\frac{|\lambda_0|}{2}\right),
    \]
    which does not vanish.
    This means that $\lambda_0 \in C^0$ with $w_0 \neq 0$ is a \textit{fold singularity} of $\exp_q$ and the normal form follows from \cref{theo:whitneyfoldpleat}. The subset of $\lambda_0 \in C^0$ such that $w_0 = 0$ is clearly nowhere dense in $C^0$.

    When $\lambda_0 \in C^1$, the equation
    \[
    \diff_{\lambda_0}f_1 \left[u_0 \partial_v - v_0 \partial_u\right] = 0
    \]
    holds, and therefore we have that
    \[
    \dim \bigl[ \Kern(\diff_{\lambda_0} \exp_q) \cap \T_{\lambda_0}(\mathrm{Conj}_q(\mathrm{SU}(2)) \bigr] = 1
    \]
    for all $\lambda_0 \in C^1$. This type of singularity has not been previously discussed in this work and it must be constructed by hand, following the idea from \cite[Theorem 3.3.]{warner1965}.
    
    We begin by choosing coordinates $\xi = (\xi_1, \xi_2, \xi_3) : \mathcal{U} \to \R^3$ adapted to the 2-dimensional manifold $C^1$ and centered at $\lambda_0$, that is to say $\xi(\lambda_0) = 0$ and 
    \[
    C^1 \cap \mathcal{U} = \left\{\overline{\lambda}_0\in \T^*_q(\mathrm{SU}(2)) \cap \mathcal{U}  \mid \xi_1(\overline{\lambda}_0) = 0 \right\}.
    \]The family of spaces given by $\mathrm{Ker}(\diff_{\lambda_0} \exp_q)$ forms an involutive distribution of $C^1$ and by Frobenius' theorem the coordinate chart $(\mathcal{U}, \xi)$ can be chosen such that the integral manifolds of this distribution are the submanifolds
    \[
    \left\{ \overline{\lambda}_0 \in C^1 \cap \mathcal{U} \mid \xi_1(\overline{\lambda}_0) = 0, \xi_2(\overline{\lambda}_0) = c \right\},
    \] 
    where $c$ is a constant. We also choose a coordinate chart $\eta = (\eta_1, \eta_2, \eta_3) : \mathcal{V} \to \R^3$ centered at $\exp_q(\lambda_0)$ that is also adapted to the action of $\exp_q$ on $C^1$ and on the integral manifolds of $C^1$. Firstly, it is chosen such that
    \begin{equation}
        \label{slice1}
        \exp_q(C^1 \cap \mathcal{U}) \subseteq \left\{ p \in \mathrm{SU}(2) \cap \mathcal{V} \mid \eta_1(p) = 0 \right\}.
    \end{equation}
    Secondly, each integral manifold of the distribution is sent to a point in $\mathrm{SU}(2)$. The chart $(\eta, \mathcal{V})$ can therefore also be chosen such that
    \begin{multline}
    \label{slice2}
    \exp_q\left(\left\{ \overline{\lambda}_0 \in \T^*_q(\mathrm{SU}(2)) \cap \mathcal{U} \mid \xi_1(\overline{\lambda}_0) = 0, \xi_3(\overline{\lambda}_0) = 0 \right\}\right)  \\
    = \left\{ p \in \mathrm{SU}(2) \cap \mathcal{V} \mid \eta_1(p) = 0, \eta_3(p) = 0 \right\}.
    \end{multline}
    Thirdly, it finally satisfies
    \begin{equation}
        \label{slice3} \diff_{\lambda_0} \exp_q \left( \frac{\partial}{\partial \xi_1}\right) =  \frac{\partial}{\partial \eta_1}, \ \diff_{\lambda_0} \exp_q \left( \frac{\partial}{\partial \xi_2}\right) = \frac{\partial}{\partial \eta_2}, \text{ and } \diff_{\lambda_0} \exp_q \left( \frac{\partial}{\partial \xi_3}\right) = 0.
    \end{equation}
With this choice of coordinates, specifically because of of \cref{slice3}, the Taylor's expansion at $(0, 0, 0)$ of $(\exp_q)_{\xi \eta} := \eta \circ \exp_q \circ \xi^{-1}$, writes as:
\[
(\exp_q)^1_{\xi \eta} = x + R_1, \ (\exp_q)^2_{\xi \eta} = y + R_2, \text{ and } (\exp_q)^3_{\xi \eta} = R_3,
\]
where the integral remainders $R_1, R_2,$ and $R_3$ have a zero of order 2 at $(0, 0, 0)$. A new coordinate chart $\omega = (\omega_1, \omega_2, \omega_3)$ in the neighbourhood of $\lambda_0$ can thus be introduced as
\begin{equation}
\label{omegacoord}
    \omega_1 := \eta_1 \circ \exp_q = \xi_1 + R_1 \circ \xi, \ \omega_2 := \eta_2 \circ \exp_q = \xi_2 + R_2 \circ \xi, \ \omega_3 := \xi_3.
\end{equation}
This indeed forms a valid charts because the order of vanishing of $R_1 \circ \xi$ and $R_2 \circ \xi$ is no less than 2. 

Consider now the projection $\pi_{3, 2} : \R^3 \to \R^3 : (x, y, z) \mapsto (x, y, 0)$ and define new coordinates $\Upsilon = (\Upsilon_1, \Upsilon_2, \Upsilon_3)$ around $\exp_q(\lambda_0)$ as
\begin{equation}
\label{upsiloncoord}
    \Upsilon_1 = \eta_1, \ \Upsilon_2 = \eta_2, \, \Upsilon_3 = \eta_3 - \eta_3 \circ \exp_q \circ \omega^{-1} \circ \pi_{3, 1} \circ \eta.
\end{equation}
Again, this is a well-defined coordinates chart because the order of vanishing at $(0, 0, 0)$ of $R_3$ is greater than 2. The reason for this choice of coordinate chart will become apparent later. From \cref{slice1}, \cref{slice2}, \cref{omegacoord}, and \cref{upsiloncoord}, it is clear that
\[
\Upsilon_3 \circ \exp_q(\overline{\lambda}_0) = 0 \text{ whenever } \omega_1(\overline{\lambda}_0) = 0.
\]
There must exists a function $\chi_3$ such that $\Upsilon_3 \circ \exp_q = \omega_1 \chi_3 $ on a neighbourhood of $\lambda_0$, satisfying $\chi_3(\lambda_0) = 0$ because of \cref{omegacoord} and since the order of $R_3$ at $(0, 0, 0)$ is no less than 2. Finally, let $\chi_1 := \omega_1$, $\chi_2 := \omega_2$. It remains to prove that $\chi = (\chi_1, \chi_2, \chi_3)$ is a well-defined coordinate chart on a neighbourhood of $\lambda_0$ since if that's the case, then the expression of $\exp_q$ in the coordinates $\chi$ and $\Upsilon$ is
\[
(\exp_q)_{\chi \Upsilon} (x, y, z) = (x y, y, z).
\]
The remaining of this proof will therefore focus on establishing that the matrix $(\partial \chi_i / \partial \omega_j (\lambda_0))$ is invertible, which here is equivalent to proving that $\partial \chi_3 / \partial \omega_3 (\lambda_0) = \partial^2 (\Upsilon_3 \circ \exp_q)/\partial \omega_1 \partial \omega_3 (\lambda_0) \neq 0$. This property that needs to be established is on the second derivative of $\exp_q$. Let us start by compute all the other first and second partial derivatives. From \cref{omegacoord}, we can already say that
\[
\frac{\partial (\Upsilon_1 \circ \exp_q)}{\partial \omega_1}(\lambda_0) = 1, \text{ and } \frac{\partial (\Upsilon_2 \circ \exp_q)}{\partial \omega_2}(\lambda_0) = 1,
\]
while the partial derivatives of second order of $\Upsilon_1 \circ \exp_q$ and $\Upsilon_2 \circ \exp_q$ are zero. Since $\lambda_0$ is a zero of order 2 of $\Upsilon_3 \circ \exp_q$, we have that all the partial derivatives of first order of $\Upsilon_3 \circ \exp_q$ are zero at $\lambda_0$.
For partial derivatives of second order, we compute first
\[
\frac{\partial^2 (\Upsilon_3 \circ \exp_q)}{\partial \omega_2^2}(\lambda_0) = \frac{\partial^2 (\Upsilon_3 \circ \exp_q)}{\partial \omega_3^2}(\lambda_0) = \frac{\partial^2 (\Upsilon_3 \circ \exp_q)}{\partial \omega_2 \partial \omega_3}(\lambda_0) =0,
\]
from the fact that $\Upsilon_3 \circ \exp_q = \omega_1 \chi_3$ and $\omega_1(\lambda_0) = 0$. Secondly, the specific choice in \cref{upsiloncoord} implies that
\[
\frac{\partial^2 (\Upsilon_3 \circ \exp_q)}{\partial \omega_1^2}(\lambda_0) = \frac{\partial^2 (\Upsilon_3 \circ \exp_q)}{\partial \omega_1 \partial \omega_2}(\lambda_0) =0.
\]
In order to conclude, consider the second-order differential (see \cite[Section 1.26]{warnerbook}) of $\exp_q$ at $\lambda_0$:
\begin{equation}
    \label{eq:secondorderexp}
    \diff_{\lambda_0}^2 \exp_q : \T_{\lambda_0}^2(\T^*_q(\mathrm{SU}(2))) \to \T^2_{\exp_q(\lambda_0)}(\mathrm{SU}(2)),
\end{equation}
defined as $\diff_{\lambda_0}^2 \exp_q[v](f) := v(f \circ \exp_q)$ for all smooth function $f \in \mathcal{C}^{\infty}$ and all second-order tangent vectors $v \in \T_{\lambda_0}^2(\T^*_q(\mathrm{SU}(2)))$. Now, a short computation (see also \cite[Lemma 2.1.]{warner1965}) shows that if $\mathfrak{e}(\lambda_0)$ is the Euler vector field at $\lambda_0$ and $v \in \mathrm{Ker}(\diff_{\lambda_0} \exp_q)$, that is, by \cref{eq:kernelconjugatelocus}, $v = p_a(0) E_a(0) + p_b(0) E_b(0) + p_c(0) E_c(0)$ where $(p(t), x(t))$ satisfies \cref{eq:jacobifields} with $x(0) = x(1) = 0$, then
\begin{multline*}
    \label{slice2}
    \diff_{\lambda_0}^2 \exp_q[ \mathfrak{e} v (\lambda_0)] = \dot{x}_a(1) X_a(1) + \dot{x}_b(1) X_b(1) + \dot{x}_c(1) X_c(1) \\ + \diff_{\lambda_0}\exp_q\left[\T_{\lambda_0}(\T^*_q(\mathrm{SU}(2)))\right].
    \end{multline*}

We claim that $\diff_{\lambda_0}^2 \exp_q[ \mathfrak{e} v (\lambda_0)] \notin \diff_{\lambda_0}\exp_q\left[\T_{\lambda_0}(\T^*_q(\mathrm{SU}(2)))\right]$. By \cref{eq:jacobifieldssu2explicit} and \cref{prop:conjlocuskernelSU2} we can make an explicit computation: if $v = u_0 \partial_v - v_0 \partial_u \in \Kern(\diff_{\lambda_0} \exp_q)$, then $\dot{x}_a(1) X_a(1) + \dot{x}_b(1) X_b(1) + \dot{x}_c(1) X_c(1) = X_a(1)$. Now, with \cref{prop:imageexp}, we find that
\[
\diff_{\lambda_0}\exp_q\left[\T_{\lambda_0}(\T^*_q(\mathrm{SU}(2)))\right] = \vspan\left\{ |\lambda_0|^2 X_b(1) - X_c(1) \right\},
\]
proving the claim.

Therefore, calling $\alpha : \T^2_{\exp_q(\lambda_0)}(\mathrm{SU}(2)) \to \T_{\exp_q(\lambda_0)}(\mathrm{SU}(2))$ the projection of the space of second-order tangent vector onto the space of first-order tangent vector, we must have that the linear map $\alpha \circ \diff_{\lambda_0}^2 \exp_q : \T_{\lambda_0}^2(\T^*_q(\mathrm{SU}(2))) \to \T_{\exp_q(\lambda_0)}(\mathrm{SU}(2))$ is a surjection. The matrix of this map in the basis induced by the coordinates $\omega$ around $\lambda_0$ and $\Upsilon$ around $\exp_q(\lambda_0)$ is given by the matrix of first and second partial derivatives of $\Upsilon_i \circ \exp_q$ with respect to $\omega_j$ at $\lambda_0$ and by the computations above it is equal to
\[
    \begin{pmatrix}
			1 & 0 & 0 & \cdots & 0 & 0 \\
			0 & 1 & 0 & \cdots & 0 & 0 \\
			0 & 0 & 0 & \cdots & 0 & \frac{\partial^2 (\Upsilon_3 \circ \exp_q)}{\partial \omega_1 \partial \omega_3}(\lambda_0)
    \end{pmatrix},
\]
which must have rank 3 by surjectivity. This implies that 
\[
\partial^2 (\Upsilon_3 \circ \exp_q)/\partial \omega_1 \partial \omega_3 (\lambda_0) \neq 0
\]
and the proof is complete.
\end{proof}

As a direct consequence of the normal forms identified in \cref{prop:normalformsSU2}, we can derive a result similar to \cref{prop:localnoninjGa} for $\mathrm{SU}(2)$.

\begin{proposition}
    \label{prop:localnoninjSU2}
    The exponential map $\exp_q : \T^*_q(\mathrm{SU}(2)) \to \mathrm{SU}(2)$ fails to be injective in any neighbourhood of conjugate covector $\lambda_0 \in \mathrm{Conj}_q(\mathrm{SU}(2))$.
\end{proposition}

\begin{proof}
For $\lambda_0 \in C^0$ with $w_0$, this is direct from the normal form: $\lambda_0$ is a Whitney fold of $\exp_q$. If $w_0 = 0$, the injectivity still fails by density.

When $\lambda_0 \in C^1$, the normal forms of $\exp_q$ also implies the result. Alternatively, on may argue without referencing the normal form, by noting that the integral manifolds of $\mathrm{Ker}(\diff \exp_q)$ are collapsed to points under $\exp_q$.
\end{proof}

\section{The special linear group \texorpdfstring{$\mathrm{SL}(2)$}{SL(2)}}

\label{sec:sl2}

The normal forms of the exponential map of $\mathrm{SL}(2)$ closely resemble those of $\mathrm{SU}(2)$ examined previously. This section will therefore only provide a concise overview of the analysis, without reiterating all the details.

The Lie group $\mathrm{SL}(2)$ is the group of $2 \times 2$ real matrices defined as
\[
\mathrm{SL}(2) = \left\{ 
M \in \mathds{R}^{2 \times 2}
\mid \det(M) = 1
\right\}.
\]

The Lie algebra $\mathfrak{sl}(2)$ of $\mathrm{SL}(2)$ is naturally identified with the space of traceless $2 \times 2$ real matrices:
\[
\mathfrak{sl}(2) = \left\{ 
\begin{pmatrix}
a & b \\
c & -a
\end{pmatrix}
\mid a, b, c \in \R
\right\} \cong \mathrm{T}_e(\mathrm{SL}(2)).
\]

A basis for $\mathfrak{su}(2)$ is given by $X_0, X_1, X_2$ where
\[
X_0 := 
\frac{1}{2}
\begin{pmatrix}
0 & -1 \\
1 & 0
\end{pmatrix}, \ 
X_1 := \frac{1}{2}
\begin{pmatrix}
1 & 0 \\
0 & -1
\end{pmatrix}\text{, and }
X_2 :=
\frac{1}{2} 
\begin{pmatrix}
0 & 1 \\
1 & 0
\end{pmatrix}.
\]
We also have the commutation relations 
\begin{equation}
    \label{eq:bracketrelationsSL2}
    [X_1, X_2] = - X_0, \ [X_2, X_0]= X_1, \text{ and } [X_0, X_1] = X_2.
\end{equation}

We construct a sub-Riemannian structure in the following way. Let $\mathbf{k} := \mathrm{span}\left\{X_0\right\}$ and $\mathbf{z} := \mathrm{span}\left\{X_1, X_2\right\}$. It can be seen that $\mathbf{k}$ is a (maximal) compact subalgebra while $\mathbf{z}$ is contained in the center of $\mathfrak{sl}(2)$. We equip $\mathfrak{sl}(2)$ with the bi-invariant metric $\langle \cdot, \cdot \rangle$ that turns $X_0, X_1, X_2$ into an othonormal basis. By left-translating $\mathbf{z}$ and the restriction $\langle \cdot, \cdot \rangle_{\mathbf{z}}$ of $\langle \cdot, \cdot \rangle$ onto $\mathbf{z}$, we obtain a well-defined sub-Riemannian structure on $\mathrm{SL}(2)$. This kind of sub-Riemannian structure on a Lie group is often called a $\mathbf{k} \oplus \mathbf{z}$ sub-Riemannian structure and the expression of the normal geodesics for such a structure are well-known (see \cite[Section 7.7.4]{comprehensive2020}). The metric $\langle \cdot, \cdot \rangle$ is used to identify vectors and covectors.

It is helpful to define the following functions: for $a \in \mathbb{R}$, set
\[
\mathsf{s}_a(t):=
    \begin{cases}
    \sin(\sqrt{a}t)/\sqrt{a}, & \text{if}\ a > 0 \\
      t, & \text{if}\ a = 0 \\
      \sinh(\sqrt{-a}t)/\sqrt{-a}, & \text{if}\ a < 0
    \end{cases}, \text{ and }
\mathsf{c}_a(t):=
    \begin{cases}
     \cos(\sqrt{a}t), & \text{if}\ a > 0 \\
      1, & \text{if}\ a = 0 \\
      \cosh(\sqrt{-a}t), & \text{if}\ a < 0
    \end{cases}.
\]
 
The sub-Riemannian Hamiltonian of this structure is given by
\[
H : \mathrm{T}^*(\mathrm{SL}(2)) \to \mathds{R} : (u X_1(M) + v X_2(M) + w X_0(M), M) \mapsto \frac{1}{2}(u^2 + v^2).
\]
Furthermore, when $\lambda_0 \in \T^*_q(\mathrm{SL}(2))$, define
\[
r_{\lambda_0} := w_0^2 - (u_0^2 + v_0^2).
\]
By left-invariance, it is enough to write the geodesics starting from the identity.
In this case, the normal extremal starting from $(q, \lambda_0) \in \T^*(\mathrm{SL}(2))$ with $q = 0$ and $\lambda_0 = u_0 X_1 + v_0 X_2 + w_0 X_0$ is
\begin{equation*}
    \label{eq:SL2geodesic}
    \left\{
    \begin{aligned}
    M(t) ={}& 
    \begin{pmatrix}
    m_1(t) & m_2(t) \\
    m_3(t) & m_4(t)
    \end{pmatrix}
    \\
    u(t) ={}& u_0 \cos(w_0 t) + v_0 w_0 \sin(w_0 t)  \\
    v(t) ={}& v_0 \cos(w_0 t) - u_0 w_0 \sin(w_0 t) \\
    w(t) ={}& w_0
    \end{aligned},
    \right.
\end{equation*}
where
\begin{align*}
    m_1(t) ={}& 
    \mathsf{s}_a(\tfrac{t}{2}) \left(u_0 \cos(\tfrac{w_0 t}{2}) +(w_0-v_0) \sin(\tfrac{w_0 t}{2})\right)+ \mathsf{c}_a(\tfrac{t}{2}) \cos(\tfrac{w_0 t}{2});\\
    m_2(t) ={}& \mathsf{s}_a(\tfrac{t}{2}) \left((v_0 - w_0) \cos(\tfrac{w_0 t}{2})+ u_0 \sin (\tfrac{w_0 t}{2})\right)- \mathsf{c}_a(\tfrac{t}{2}) \sin (\tfrac{w_0 t}{2}); \\
    m_3(t) ={}& \mathsf{s}_a(\tfrac{t}{2}) \left((v_0 + w_0) \cos (\tfrac{w_0 t}{2})+ u_0 \sin (\tfrac{w_0 t}{2})\right)- \mathsf{c}_a(\tfrac{t}{2}) \sin (\tfrac{w_0 t}{2}); \\
    m_4(t) ={}& \mathsf{s}_a(\tfrac{t}{2}) \left(-u_0 \cos (\tfrac{w_0 t}{2})+(v_0 + w_0) \sin (\tfrac{w_0 t}{2})\right)+ \mathsf{c}_a(\tfrac{t}{2}) \cos (\tfrac{w_0 t}{2}).
\end{align*}

We use the same coordinates induced from the frame $X_0, X_1, X_2$ on $\T^*(\mathrm{SL}(2))$ that were used for $\T^*(\mathrm{SU}(2))$. The canonical frame can also be computed for $\mathrm{SL}(2)$, as in \cref{prop:zelenkoliSU2}, with the relevant entries of the matrix $R(t)$ from \cref{zelenkoli} expressed as
\begin{align*}
    R_{aa}(t) ={}& \frac{1}{2 H} \sigma_{\lambda(t)}([\overrightarrow{H}, \overrightarrow{H}'], \overrightarrow{H}') = h_0^2 - 2 H  = r_{\lambda_0}, \text{ and } \\
    R_{cc}(t) ={}& \frac{1}{2 H} \sigma_{\lambda(t)}([\overrightarrow{H}, [\overrightarrow{H}, \overrightarrow{H}']], [\overrightarrow{H}, \overrightarrow{H}']) - \frac{1}{(2 H)^2} \sigma_{\lambda(t)}([\overrightarrow{H}, \overrightarrow{H}'], \overrightarrow{H}')^2 = 0.
\end{align*}

\begin{proposition}
    \label{prop:zelenkoliSL2}
    Let $\lambda(t)$ be the normal extremal in $\T^*(\mathrm{SL}(2))$ starting from $(q, \lambda_0) \in \mathrm{T}^*(\mathrm{SL}(2))$. The symplectic moving frame $(E_a, E_b, E_c, F_a, F_b, F_c)$ where
    \begin{equation*}
    \label{eq:SL2movingframe}
    \begin{aligned}
    E_a(t) ={}& \frac{1}{\sqrt{2 H}} \partial_\theta & F_a(t) ={}& \frac{1}{\sqrt{2 H}} \overrightarrow{H'}\\
    E_b(t) ={}& \frac{1}{\sqrt{2 H}} \mathfrak{e} & F_b(t) ={}& \frac{1}{\sqrt{2 H}} \overrightarrow{H}\\
    E_c(t) ={}& \frac{1}{\sqrt{2 H}} \partial_{h_0} & F_c(t) ={}& \frac{1}{\sqrt{2 H}} ([\overrightarrow{H'}, \overrightarrow{H}] + (w_0^2 - 2 H) \partial_\theta)\\
    \end{aligned}
\end{equation*}
is a canonical moving frame along $\lambda(t)$. More specifically, it satisfies the structural equations
\begin{equation*}
    \label{eq:SL2structuraleqODE}
    \begin{aligned}
    \dot{E}_a(t) ={}& - F_a(t) & \dot{F}_a(t) ={}& (w_0^2 - 2 H) E_a(t)\\
    \dot{E}_b(t) ={}& - F_b(t) & \dot{F}_b(t) ={}& 0 \\
    \dot{E}_c(t) ={}& E_a(t) & \dot{F}_c(t) ={}& 0\\
    \end{aligned}.
\end{equation*}
\end{proposition}

The equation for the Jacobi fields of $\mathrm{SL}(2)$ follows easily from \cref{eq:jacobifields}.

\begin{proposition}
    If $(E, F)$ denotes the symplectic frame of \cref{prop:zelenkoliSL2}, then the vector field $\mathcal{J}(t) = \sum_{k \in \{a, b, c\}} p_k(t) E_k(t) + x_k(t) F_k(t)$ along $\lambda(t)$ is a Jacobi field if and only if
    \begin{equation}
    \label{jacobifieldssl2}
    \begin{pmatrix}
			\dot{p}(t) \\
			\dot{x}(t)
    \end{pmatrix}
    =
    \begin{pmatrix}
			-C_1(t) & -R(t) \\
			C_2(t) & C_1(t)\tran
    \end{pmatrix}
    \begin{pmatrix}
			p(t) \\
			x(t)
    \end{pmatrix},
    \end{equation}
    where
    \[
    C_1(t) = \begin{pmatrix}
			0 & 0 & 1 \\
			0 & 0 & 0 \\
            0 & 0 & 0
    \end{pmatrix}, \
    C_2(t) = \begin{pmatrix}
			1 & 0 & 0 \\
			0 & 1 & 0 \\
            0 & 0 & 0
    \end{pmatrix},
    \]
    \[
    R(t) = \begin{pmatrix}
			w_0^2 - 2 H & 0 & 0 \\
			0 & 0 & 0 \\
            0 & 0 & 0
    \end{pmatrix}.
    \]
\end{proposition}

These equations are solved explicitly, with an explicit dependence on $r_{\lambda_0}$:

\begin{equation}
    \label{eq:jacobifieldssl2explicit}
    \left\{
    \begin{aligned}
    p_a(t) ={}& p_{a0} \mathsf{c}_{r_{\lambda_0}}(t) - \left(a x_{a0} + p_{c0} \right) \mathsf{s}_{r_{\lambda_0}}(t)\\
    p_b(t) ={}& p_{b0}\\
    p_c(t) ={}& p_{c0} \\
    x_a(t) ={}& \frac{\left({r_{\lambda_0}} x_{a0} + p_{c0} \right) \mathsf{c}_{r_{\lambda_0}}(t) - p_{c0}}{a} +  p_{a0} \mathsf{s}_{r_{\lambda_0}}(t) \\
    x_b(t) ={}& p_{b0} t + x_{b0}\\
    x_c(t) ={}& x_{c0} + \frac{p_{a0}(1 - \mathsf{c}_{r_{\lambda_0}}(t)) - p_{c0} t + ({r_{\lambda_0}} x_{a0} + p_{c0}) \mathsf{s}_{r_{\lambda_0}}(t)}{a} \\
    \end{aligned}
    \right.
\end{equation}

As in the proof \cref{prop:conjlocuskernelSU2}, the critical points of $\mathrm{SL}(2)$ can be characterised, as well as the corresponding kernel, by studying the linear system of equations
\begin{equation*}
    \begin{pmatrix}
			\mathsf{s}_{r_{\lambda_0}}(1) & 0 & \frac{\mathsf{c}_{r_{\lambda_0}}(1)-1}{{r_{\lambda_0}}} \\
			0 & 1 & 0 \\
			\frac{1-\mathsf{c}_{r_{\lambda_0}}(1)}{{r_{\lambda_0}}} & 0 & \frac{\mathsf{s}_{r_{\lambda_0}}(1)-1}{{r_{\lambda_0}}} \\
    \end{pmatrix}
    \begin{pmatrix}
			p_{a0} \\
			p_{b0} \\
			p_{c0}
    \end{pmatrix}
    =
    \begin{pmatrix}
			0 \\
			0 \\
			0
    \end{pmatrix}.
    \end{equation*}

One interesting feature here is that the initial covectors with ${r_{\lambda_0}} \leq 0$ are never conjugate, because the determinant of the matrix above is then equal to
\[
    \sinh \left(\frac{\sqrt{-{r_{\lambda_0}}}}{2}\right) \left(\sqrt{-{r_{\lambda_0}}} \cosh \left(\frac{\sqrt{-{r_{\lambda_0}}}}{2}\right)-2 \sinh \left(\frac{\sqrt{-a{r_{\lambda_0}}}}{2}\right)\right) \neq 0.
\]
This can be easily seen from the properties of the hyperbolic functions.

\begin{proposition}
    \label{prop:conjlocuskernelSL2}
    The covector $\lambda_0 \in \T^*_q(\mathrm{SL}(2))$ is a conjugate covector of $q \in \mathrm{SL}(2)$ with $H(q, \lambda_0) \neq 0$ if and only if 
    \[
    {r_{\lambda_0}} > 0, \text{ and }
    \sin \left(\frac{\sqrt{{r_{\lambda_0}}}}{2}\right) \left(\sqrt{{r_{\lambda_0}}} \cos \left(\frac{\sqrt{{r_{\lambda_0}}}}{2}\right)-2 \sin \left(\frac{\sqrt{{r_{\lambda_0}}}}{2}\right)\right)  = 0.
    \]
    Furthermore, the conjugate covectors $\lambda_0 \in \mathrm{Conj}_q(\mathrm{SL}(2))$ are all of order one, and
    \begin{equation*}
    \label{eq:kernelexpSL2}
    \Kern(\diff_{\lambda_0} \exp_q)=  \vspan \left\{ \sqrt{{r_{\lambda_0}}} \cos \left(\frac{\sqrt{{r_{\lambda_0}}}}{2}\right) (u_0 \partial_v - v_0 \partial_u) + 4 \sin \left(\frac{\sqrt{{r_{\lambda_0}}}}{2}\right)  \partial_w \right\}.
    \end{equation*}
\end{proposition}

Similarly to the proof of \cref{prop:conjSU2submanifold}, we can use the implicit function theorem on the functions
\begin{align*}
    f_0(\lambda_0) &= \sqrt{{r_{\lambda_0}}} \cos \left(\frac{\sqrt{{r_{\lambda_0}}}}{2}\right)-2 \sin \left(\frac{\sqrt{{r_{\lambda_0}}}}{2}\right) \\
    f_1(\lambda_0) &= \sin \left(\frac{\sqrt{{r_{\lambda_0}}}}{2}\right)
\end{align*}
and find that
\begin{equation*}
    \label{eq:tangentconjlocusSL2C0}
    \diff_{\lambda_0} f_0 = \frac{1}{2} \sin \left(\frac{\sqrt{{r_{\lambda_0}}}}{2} \right) (u_0 \diff u + v_0 \diff v - w_0 \diff w) \neq 0 \text{ in } C^0
\end{equation*}
\begin{equation*}
    \label{eq:tangentconjlocusSL2C1}
    \diff_{\lambda_0} f_1 = -\frac{1}{2 \sqrt{{r_{\lambda_0}}}}\cos \left(\frac{\sqrt{{r_{\lambda_0}}}}{2} \right) (u_0 \diff u + v_0 \diff v - w_0 \diff w) \neq 0 \text{ in } C^1
\end{equation*}
This establishes that the conjugate locus of $\mathrm{SL}(2)$ is an hypersurface of $T_q^*(\mathrm{SL}(2))$.

\begin{proposition}
\label{prop:conjSL2submanifold}
For all $q \in \mathrm{SL}(2)$, the conjugate locus $\mathrm{Conj}_q(\mathrm{SL}(2))$ is a submanifold of $\T^*_q(\mathrm{SL}(2))$ of codimension 1.
\end{proposition}

Finally, the methods of singularity theory can be applied to find the same normal forms of the exponential map of $\mathrm{SL}(2)$ as was done in \cref{prop:normalformsSU2}.

\begin{theorem}
    \label{prop:normalformsSL2}
    The exponential map $\exp_q : \T^*(\mathrm{SL}(2)) \to \mathrm{SL}(2)$ in the neighbourhood of $\lambda_0 \in \mathrm{Conj}_q(\mathrm{SL}(2))$ is equivalent to $f : \R^3 \to \R^3 : (x, y, z) \to (x^2, y, z)$ when $\lambda_0 \in C^0$ and to $f : \R^3 \to \R^3 : (x, y, z) \to (x z, y, z)$ when $\lambda_0 \in C^1$.
\end{theorem}

Since a covector $\lambda_0 \in \T^*_{q}(\mathrm{SL}(2))$ with $w_0 = 0$ is never conjugate, all conjugate covectors of $\lambda_0 \in C^0$ are fold singularities of $\exp_q$, not just a dense subset of $C^0$.

\begin{proposition}
    \label{prop:localnoninjSL2}
    The exponential map $\exp_q : \T^*_q(\mathrm{SU}(2)) \to \mathrm{SU}(2)$ fails to be injective in any neighbourhood of conjugate covector $\lambda_0 \in \mathrm{Conj}_q(\mathrm{SU}(2))$.
\end{proposition}

\section{Conclusion and open problems}

\label{sec:conclusion}

As seen in the previous sections, it can be concluded that an accurate analysis of the sub-Riemannian exponential map near critical points is indeed possible, at least for some specific examples.

Let's mention a type of critical point that we haven't discussed yet. It was proven in \cite[Theorem 1]{borza2022}, generalising the continuity property in Riemannian or Finsler geometry \cite[Definition 1 (R3)]{warner1965}, that for each non-trivial covector $\lambda_0$ conjugate to $q \in M$, there exists a convex neighborhood $\mathcal{U}$ of $\lambda_0$ in $\T^*_q(M)$ such that the number of singularities of the exponential map $\exp_q$ (counted
with multiplicities) on $r \cap \mathcal{U}$, for each straight line $r$ passing through 0 which intersects $\mathcal{U}$, is constant and equals $\dim(\mathrm{Ker}(\diff_{\lambda_0} \exp_q))$. Assuming the absence of abnormal subsegments, one of two things must then be true: either there exists a neighbourhood of $\lambda_0$ such that each straight line $r$
contains at most one point in $\mathcal{U}$ which is a conjugate covector, or it does not. The first type of singularity is referred to as a regular conjugate covector, while the second is known as a singular or branching conjugate covector (see \cite[Section 3]{warner1965}). The set of regular (resp. singular) conjugate values forms the regular (resp. singular) conjugate locus. It is established in \cite[Theorem 3.1]{warner1965} that for any Riemannian or Finsler manifold, the regular conjugate locus is a generic subset of $\mathrm{Conj}_q(M)$. However, it is not clear whether this remains true in sub-Riemannian geometry when there are geodesics that are both normal and abnormal.

\begin{openproblem}
Is the sub-Riemannian regular conjugate locus a non-empty open dense subset of the conjugate locus?
\end{openproblem}

In Riemannian or Finsler geometry, the regular conjugate locus is always a submanifold of codimension 1. This is true for the examples that we have studied here (see \cref{prop:conjlocusGa},  \cref{prop:conjSU2submanifold}, and \cref{prop:conjSL2submanifold}) but it is not known whether this is true for general sub-Riemannian manifolds, even when there are no non-trivial abnormal geodesics. In \cite[Section 3.3.]{borzanonlocal}, the problem is partially addressed in the affirmative, however, it comes with the requirement of an additional technical assumption on the singularity. 

Additionally, in Riemannian and Finsler geometry, an interesting phenomenon occurs. As long as the order of $\lambda_0$ is 2 or greater and that $\lambda_0$ is a regular singularity, the kernel of $\diff_{\lambda_0} \exp_q$ is always included in the tangent space to the conjugate locus at $\lambda_0$. This very geometric property was already investigated by Whitehead in \cite{whitehead1935} and revisited by Warner in \cite[Theorem 3.2.]{warner1965}. I do not know whether this is true in sub-Riemannian geometry. This property would be very useful for studying the normal forms of the exponential map around its singularities, as shown by the proof of \cref{prop:normalformsSU2}.

\begin{openproblem}
    In sub-Riemannian geometry, is the regular conjugate locus a submanifold of $\T^*_q(M)$? Is there an integer $k$ such that $\mathrm{Ker}(\diff_{\lambda_0} \exp_q) \subseteq \T_{\lambda_0}(\T^*_q(M))$ for any regular conjugate covector $\lambda_0$ of order $\geq k$? 
\end{openproblem}

Lastly, one could expand the scope of this paper by examining the exponential map of more general sub-Riemannian structures. In \cite{agrachev1996}, the author uses a related but different approach to compute the small time asymptotics of the exponential map in the three dimensional contact case (see also \cite[Chapter 19]{comprehensive2020}). The asymptotics contain some information about the structure of the cut and the conjugate locus and they are related to curvature invariants, as per \cite{curvature}. However, for applications such as \cref{prop:localnoninjGa}, \cref{prop:localnoninjSU2}, and \cref{prop:localnoninjSL2}, these asymptotics alone may not be sufficient.

\begin{openproblem}
    Extend the analysis of this work to more general classes of sub-Riemannian structures, such as 3D contact structures.
\end{openproblem}

One potential strategy is to initially investigate 3D left-invariant structures, the classification of which can be found in \cite[Section 17.5]{comprehensive2020}. In \cite[Section 4]{borza2021}, the analysis is done for the three-dimensional Heisenberg group, and the current paper has discussed $\mathrm{SU}(2)$ in \cref{sec:su2}, and $\mathrm{SL}(2)$ in \cref{sec:sl2}.

The exponential map failing to be injective in the neighbourhood of any critical point, as demonstrated in \cref{prop:localnoninjGa},  \cref{prop:localnoninjSU2} and \cref{prop:localnoninjSL2}, is significant because it indicates that the exponential map does not have singularities of the type $f(x) = x^3$. This extends a result of Morse and Littauer to certain sub-Riemannian structures. This property for the exponential map is known to hold in Finsler geometry (see \cite[Theorem 3.4.]{warner1965}), and the author of this text recently expanded upon this result to include a wide class of sub-Riemannian manifolds in \cite{borzanonlocal} using different methods than those presented here.

%    Bibliographies can be prepared with BibTeX using amsplain,
%    amsalpha, or (for "historical" overviews) natbib style.
\bibliographystyle{amsplain}
\bibliography{bibliography.bib}
%    Insert the bibliography data here.

\end{document}